\documentclass[11pt]{article}
\usepackage[english]{babel}
\usepackage[a4paper]{geometry}
\usepackage{amssymb,amsbsy,amsmath,amsfonts,amscd}
\usepackage{amsthm}
\usepackage{graphicx,color}
\usepackage{caption}
\usepackage{cancel}
\usepackage{adjustbox}
\usepackage{siunitx}
\usepackage{cases}
\usepackage[toc,page]{appendix}
\usepackage{comment}
\usepackage{enumitem}

\usepackage{geometry}
\usepackage{empheq}
\usepackage{pgfplots} 
\usepackage{authblk} 

\usepackage{todonotes}

\usepackage[colorlinks=true]{hyperref}

\usepackage[textfont=it,font=small,labelfont=bf]{caption}
\usepackage{subfig}

\newcommand\restr[2]{{
\left.\kern-\nulldelimiterspace 
#1 
\vphantom{\big|} 
\right|_{#2} 
}}

\newcommand{\commentout}[1]{}
\newcommand{\R}{\mathbb{R}}

\newcommand{\proj}{\hat P^1_h}

\newcommand {\es} {\varepsilon}

\newcommand {\sg} {\sigma}
\newcommand {\vp} {\varphi}

\newcommand {\Chi} {{\bf \raise 2pt \hbox{$\chi$}} }
\newcommand {\dt}   {{\Delta t}}

\newcommand {\cae} { {\mathcal E} }

\newcommand {\f}   {\frac}
\newcommand {\p}   {\partial}
\newcommand {\ov}  {\overline}
\newcommand {\undel}  {\underline}
\newcommand{\norm}[1]{\left\lVert#1\right\rVert}
\newcommand{\abs}[1]{\left\lvert#1\right\rvert}
\newcommand*{\dd}{\mathop{\kern0pt\mathrm{d}}\!{}}

\newcommand{\seconde}{{\prime\prime}}

\newcommand{\scal}[2]{\left(#1 , #2\right)}
\newcommand{\lscal}[2]{\left(#1 , #2\right)^h}

\newcommand{\kk}{{k+1}}
\newcommand{\heps}{{h,\varepsilon}}

\newcommand*{\fespace}{V^h}
\newcommand*{\bas}{\chi}

\newcommand{\ie}{\textit{i.e.}\;}

\newcommand{\eg}{\textit{e.g.}\;}
\newcommand{\etal}{\textit{et al.}\;}

\def\XXint#1#2#3{{\setbox0=\hbox{$#1{#2#3}{\int}$ }
\vcenter{\hbox{$#2#3$ }}\kern-.6\wd0}}

\newlength\eqnspace
\newtheorem{theorem}{Theorem}
\newtheorem{lemma}[theorem]{Lemma}
\newtheorem{definition}[theorem]{Definition}
\newtheorem{remark}[theorem]{Remark}
\newtheorem{proposition}[theorem]{Proposition}
\newtheorem{corollary}[theorem]{Corollary}

\numberwithin{equation}{section}

\numberwithin{equation}{section}
\numberwithin{theorem}{section} 
 \newcommand{\AP}[1]{{{#1}}}
 \newcommand{\setedges}{\mathcal{E}_h}

\title{\bfseries A nonnegativity preserving scheme for the relaxed Cahn-Hilliard equation with single-well potential and degenerate mobility }

\author[1,2]{Federica Bubba \thanks{Deceased on june 25th 2020}}
\author[1]{Alexandre Poulain \thanks{Current affiliation: CNRS, Univ. Lille, UMR 8524 - Laboratoire Paul Painlev\'e, F-59000 Lille, France. Email: \texttt{alexandre.poulain@univ-lille.fr}} \thanks{A.P. has received funding from the European Research Council (ERC) under the European Union's Horizon 2020 research and innovation programme (grant agreement No 740623)}}

\affil[1]{\small Sorbonne Universit\'{e}, CNRS,  Universit\'{e} de Paris, Inria, Laboratoire Jacques-Louis Lions, F-75005 Paris, France.}
\affil[2]{\small MOX, Dipartimento di Matematica, Politecnico di Milano, via Bonardi 9, 20133 Milano, Italy.}
\date{\today}
\begin{document}
\maketitle

\begin{center}
This is the preprint version of the published article:
Bubba, F., \& Poulain, A. (2022). A nonnegativity preserving scheme for the relaxed Cahn–Hilliard equation with single-well potential and degenerate mobility. ESAIM: Mathematical Modelling and Numerical Analysis, 56(5), 1741-1772. doi: https://doi.org/10.1051/m2an/2022050 (published under CC BY 4.0 license: https://creativecommons.org/licenses/by/4.0/).
\end{center}

\begin{abstract}
  We propose and analyze a finite element approximation of the relaxed Cahn-Hilliard equation with singular single-well potential of Lennard-Jones type and degenerate mobility that is energy stable and nonnegativity preserving. The Cahn-Hilliard model has recently been applied to model evolution and growth for living tissues. Although the choices of degenerate mobility and singular potential are biologically relevant, they induce difficulties regarding the design of a numerical scheme. We propose a finite element scheme, and we show that it preserves the physical bounds of the solutions thanks to an upwind approach adapted to the finite element method. We propose two different time discretizations leading to a non-linear and a linear scheme. Moreover, we show the well-posedness and convergence of solutions of the non-linear numerical scheme. Finally, we validate our scheme by presenting numerical simulations in one and two dimensions.
\end{abstract}
\maketitle

\noindent{\makebox[1in]\hrulefill}\newline
2010 \textit{Mathematics Subject Classification.} 35Q92, 65M60,  35K55, 35K65, 35K35
\newline\textit{Keywords and phrases.} Degenerate Cahn-Hilliard equation, single-well potential, continuous Galerkin finite elements, upwind scheme, convergence analysis.

\section{Introduction}
\label{sec:intro}
Being of fourth order, the Cahn-Hilliard equation does not fit usual softwares for finite elements. To circumvent this difficulty a relaxed version has been proposed in \cite{poulain_relaxation_2019} and the presentation of a finite element numerical scheme that preserves the physical properties of the solutions is the purpose of the present work.
The relaxed version of the Cahn-Hilliard (RDCH in short) equation reads
\begin{equation}
\begin{cases}
\dfrac{\partial n}{\partial t} = \nabla \cdot \left(b(n) \, \nabla\left( \vp + \psi^\prime_+(n) \right) \right),\\\\
-\sg \, \Delta \vp + \vp = -\gamma \, \Delta n + \psi^\prime_-\left(n-\dfrac{\sg}{\gamma}\vp\right),
\end{cases} t >0,  \, x \in \Omega,
\label{eq:CH-relax}
\end{equation}
and is set in a regular bounded domain $\Omega \subset \R^d$ with $d = 1, \, 2, \, 3$. It describes the evolution in time of the (relative) volume fraction $n \equiv n(t,x)$ of one of the two components in a binary mixture. The system is equipped with nonnegative initial data
\begin{equation*}
n(0,x) = n^0(x) \in H^1(\Omega), \qquad 0 \le n^0(x) < 1, \quad x \in \Omega,
\end{equation*}
and with zero-flux boundary conditions on the boundary $\partial \Omega$ of $\Omega$
\begin{equation*}
\dfrac{\p (n-\frac{\sigma}{\gamma} \vp)}{\p \nu} =  b(n) \, \dfrac{\p \left( \vp + \psi^\prime_+(n)\right)}{\p \nu} = 0, \quad t > 0, \, x \in \p \Omega,
\end{equation*}
where $\nu$ is the unit normal vector pointing outward $\p \Omega$.\newline
\AP{
As pointed at the end of \cite{poulain_relaxation_2019}, System~\eqref{eq:CH-relax} admits a rewriting that can proved to be useful for numerical simulations. Indeed, in the following of this article we use the fact that defining $w = n- \f \sg \gamma \vp$, System~\eqref{eq:CH-relax} can be rewritten as 
\begin{equation}
\begin{cases}
\dfrac{\partial n}{\partial t} = \nabla \cdot \left(b(n) \, \nabla\left( \f\gamma\sg n + \psi^\prime_+(n) - \f\gamma\sg w \right) \right),\\\\
-\sg \, \Delta w  + \f\gamma\sg\psi^\prime_-\left(w\right)+ w   = n,
\end{cases} t >0,  \, x \in \Omega.
\label{eq:CH-relax-equiv}
\end{equation}
}
System~\eqref{eq:CH-relax} was proposed in~\cite{poulain_relaxation_2019} as an approximation, in the asymptotic regime whereby the \emph{relaxation parameter} $\sigma$ vanishes (\emph{i.e.}, $\sigma \to 0$), of the fourth order Cahn-Hilliard equation~\cite{cahn_spinodal_1961, cahn_free_1958}. The Cahn-Hilliard (CH) equation describes spinodal decomposition phenomena occurring in binary alloys after quenching: an initially uniform mixed distribution of the alloy undergoes phase separation and a two-phase inhomogeneous structure arises. In its original form, the Cahn-Hilliard equation is written in the form of an evolution equation for $n$:
\begin{equation}
\p_t n = \nabla \cdot \left( b(n) \nabla \left( \psi^\prime(n) - \gamma \Delta n \right) \right), \quad t>0, \, x \in \Omega,
\label{eq:CH}
\end{equation}
with \AP{$n \in [0,1]$, where the states $n\equiv 0$ and $n\equiv 1$} denote the two pure phases arising after the mixture has undergone the phase separation process. Writing the flux as $\mathbf{J} = -b(n) \nabla \left( \f{\delta \cae[n]}{\delta n}\right)$, Equation~\eqref{eq:CH} can be interpreted as the conservative gradient flow of the free energy functional
\begin{equation*}
\cae[n](t) := \int_\Omega \left(\f{\gamma}{2} \left|\nabla n \right|^2 + \psi(n) \right) \, \text{d}x.
\end{equation*}
The \emph{homogeneous free energy} $\psi$ describes repulsive and attractive interactions between the two components of the mixture while the regularizing term $\frac{\gamma}{2}\left|\nabla n \right|^2$ accounts for partial mixing between the pure phases, \AP{leading to a \emph{diffuse interface} separating the states $n\equiv 0$ and $n\equiv1$}, of thickness proportional to $\sqrt{\gamma}$. The parameter $\gamma >0$ is related to the surface tension at the interface (see, \emph{e.g.},~\cite{miranville_2017}) and the function $b$ is called \emph{mobility}.\newline
In most of the literature, $\psi$ is a double-well logarithmic potential, often approximated by a smooth polynomial function, with minimums located at the two attraction points that represent pure phases (see, \emph{e.g.},~\cite{cherfils_2011,elliott_cahn-hilliard_1986,elliott_1989}). The mobility can be either constant~\cite{elliott_cahn-hilliard_1986, elliott_1989} or degenerate at the pure phases~\cite{barrett_finite_1999,elliott_cahn-hilliard_1996}. We refer to the introductory chapters~\cite{elliott_1989_chapter, novick-cohen_chapter_2008} and to the recent review~\cite{miranville_2017} for an overview of the derivation of the Cahn-Hilliard equation, its analytical properties and its variants. \par

Recently, the Cahn-Hilliard equation has been considered as a phenomenological model for the description of cancer growth; see, for instance,~\cite{agosti_self-organised_nodate, chatelain_2011,wise_three-dimensional_2008}. In this context, $n$ represents the volume fraction of the tumor in a two-phase mixture containing cancerous cells and a liquid phase, such as water and other nutrients. In biological contexts, a double-well potential appears to be nonphysical. In fact, as suggested by Byrne and Preziosi in~\cite{byrne_modelling_2004}, a single-well potential of Lennard-Jones type allows for a more suitable description of attractive and repulsive forces acting in the mixture. Following this intuition and building upon previous works~\cite{agosti_cahn-hilliard-type_2017, chatelain_2011}, in this paper we consider a single-well homogeneous free energy $\psi: [0,1) \to \R$, \AP{ that can be decomposed in a convex and a concave part $\psi_\pm$ such that
\begin{equation}
\psi(n) = \psi_+(n) + \psi_-(n),  \qquad \pm  \psi_\pm''(n) \geq 0, \qquad 0\leq n < 1.
\label{eq:psi-dec}
\end{equation}   
Additionally, we consider a singularity at $n=1$ to represent saturation by one phase (see \cite{byrne_modelling_2004}).
Hence, the potential is called {\em single-well logarithmic} and the singularity is contained in the convex part of the potential. Furthermore, we assume that 
\begin{equation}
\psi_+ \in C^2\big([0,1) \big), \quad \psi'_+(1)=\infty,
\label{eq:psi-plus}
\end{equation}
and we extend the smooth concave part defined on $[0, 1]$ to $\R$ with 
\begin{equation}
\psi_- \in C^3(\R) \qquad \psi_-, \; \psi^\prime_-,\;  \psi''_-,\;  \psi'''_- \quad  \text{are bounded and } \f{\sg}{\gamma} ||\psi_-^{\prime\prime}||_\infty < 1.
\label{eq:psi-minus}
\end{equation}
The latter assumption is necessary to bound the energy of the system from below. 
In particular, the example of potential that we have in mind is, for $n\in [0, 1)$,
\begin{equation}
    \psi_+(n) = -(1-n^*)\log(1-n) - \frac{n^3}{3}, \qquad \psi_-(n) = -(1-n^*)\frac{n^2}{2} -(1-n^*)n + k,
    \label{eq:potential}
\end{equation}
where in this case $\psi_+$ is convex if $n^*\le 0.7$.}
In the above form, $\psi$ models cell-cell attraction at small densities ($\psi^{\prime}(\cdot) < 0$ for $0<n \leq n^\star$ and $\psi^{\prime}(0) = 0$) and repulsion in overcrowded zones ($\psi^{\prime}(\cdot) > 0$ for $n \geq n^\star$); \emph{cf}. Figure~\ref{fig:single-well}. The quantity $n^\star>0$ represents the value of the cellular density at which repulsive and attractive forces are at equilibrium. With a potential of the form~\eqref{eq:potential}, the pure phases are represented by the states $n = 0$ and $n = 1$, where $n = 1$ is a singularity for $\psi$ in such a way to avoid overcrowding.
\AP{
In this work, we consider a degenerate mobility $b\in C^1([0,1]; \R^+)$, such that
\begin{equation}
b(0)=b(1)= 0,  \qquad b(n) >0 \text{ for }0<n< 1. 
\label{eq:assb}
\end{equation}   
Furthermore, the admissible mobility functions that we consider admit the decomposition
\[
    b(n) = b_1(n) b_2(n),
    \] 
with the extension on $\R$ defined by
\begin{equation}
    b_1(n) > 0, \text{ for } n>0, \quad b_1(n) = 0, \text{ for } n\le 0, 
    \label{eq:b1-def}
\end{equation}
    and
\begin{equation}
    \quad b_2(n) >0, \text{ for } n<1, \quad b_2(n) = 0, \text{ for } n\ge1. 
    \label{eq:b2-def}
\end{equation}
The typical expression in the applications we have in mind is 
\begin{equation}
b(n)=n (1-n)^2. 
\label{eq:mobility}
\end{equation}
Therefore, we can easily see that, considered as transport equations, System~\eqref{eq:CH-relax} imposes formally the property $0 \leq n \leq 1$. We also assume an additional property that there is some cancellation at $1$ such that 
\begin{equation}
b(\cdot) \psi^{\prime\prime}_+(\cdot) \in C(\R; \R^+).
\label{eq:assbpsi}
\end{equation}   
For the examples of mobility and potential we described above this assumption is satisfied.
}

The Cahn-Hilliard equation~\eqref{eq:CH} with the logarithmic single-well potential defined in~\eqref{eq:potential} and a mobility given by \eqref{eq:mobility} has been studied by Agosti~\textit{et al.} in~\cite{agosti_cahn-hilliard-type_2017}, where the authors prove well-posedness of the equation for $d\le 3$. For the relaxed version of the Cahn-Hilliard equation~\eqref{eq:CH-relax}, well-posedness and convergence to the original Cahn-Hilliard equation as $\sg\to 0$ have been proved in~\cite{poulain_relaxation_2019}. 

\AP{
We also mention here some important variants of the Cahn-Hilliard equation that have been used in the context of the modelling of tumors and in which the potential of interaction considered is of double-well form. In particular, Garcke \etal~\cite{garcke-2016} proposed a Cahn-Hilliard-Darcy system that models tumor growth and in which the cells can move due to chemotaxis, attractive and repulsive forces. This model also comprises the effect that cells are crawling in a porous medium (\ie the extra-cellular matrix). This effect is represented by a velocity term defined by Darcy's law. The tumor cells can also proliferate and die depending on the availability of nutrients. Garcke \etal~\cite{garcke_multiphase_2017} extended the model to the multiphase case to study the effect of necrosis.
In order to take into account the viscosity of the tissue, and since it is not a valid approximation to consider tissues as a porous medium,  Ebenbeck \etal~\cite{ebenbeck_cahn-hilliard-brinkman_2018} proposed a Cahn-Hilliard-Brinkman model in which Brinkman's law gives the flow velocity this time.     
}

\paragraph*{Summary of previous results and specific difficulties.} Numerous numerical methods have been developed to solve the Cahn-Hilliard equation~\eqref{eq:CH} with smooth and/or logarithmic double-well potential as well as with constant or degenerate mobility. 
Generally, a numerical scheme for the Cahn-Hilliard equation is evaluated by several aspects: \textit{i)} its capacity to keep the energy dissipation (energy stability) and the physical bounds of the solutions; \textit{ii)} if it is convergent, and if error bounds can be established; \textit{iii)} its efficiency; \textit{iv)} its implementation simplicity.
To meet the first point concerning the energy stability, several implicit schemes have been proposed. The main drawback of these methods is the necessity to use an iterative method to solve the resulting nonlinear system. To circumvent this issue, unconditionally energy-stable schemes have been proposed based on the splitting of the potential in a convex and a non-convex part. This idea comes from Eyre \cite{eyre_unconditionally_1997} and leads to unconditionally energy-stable explicit-implicit (i.e. semi-implicit) approximations of the model. For references on all the previous numerical methods discussed above, we refer the reader to the review paper \cite{Tierra-numericalCH-2015}.

For finite element approximations, most of these results are based on the second-order splitting
\begin{equation}
\begin{cases}
\p_t n = \nabla \cdot \left( b(n) \nabla \mu \right) ,\\
\mu = -\gamma \Delta n + \psi^\prime(n),
\end{cases}
\label{eq:CH-splitting}
\end{equation}
where, $\mu$ is called chemical potential; see, \emph{e.g.}, \cite{agosti_cahn-hilliard-type_2017, barrett_finite_1999, elliott_1989}.\newline
In~\cite{elliott_cahn-hilliard_1986}, Elliot and Songmu propose a finite element Galerkin approximation for the resolution of~\eqref{eq:CH} with a smooth double-well potential and constant mobility. The more challenging case of a degenerate mobility and singular potentials has been considered by Barrett \textit{et al.} in~\cite{barrett_finite_1999}, where the authors propose a finite element approximation which employs the second-order splitting~\eqref{eq:CH-splitting}. In particular, the authors provide well-posedness of the finite element approximation as well as a convergence result in the one-dimensional case. Numerical methods to solve the Cahn-Hilliard equation without the splitting technique~\eqref{eq:CH-splitting} have also been suggested. For instance, in~\cite{brenner_quadratic_2012} Brenner \textit{et al.} propose a $C^0$ interior penalty method, a class of discontinuous Galerkin-type approximations.\newline
Even though a single-well potential seems more relevant for biological applications of the Cahn-Hilliard equation, very few works focus on this case. In the already mentioned~\cite{agosti_cahn-hilliard-type_2017}, Agosti and collaborators propose a finite element method to solve Equation~\eqref{eq:CH} with the homogeneous energy given by~\eqref{eq:potential} and a degenerate mobility of the form~\eqref{eq:mobility}. As the authors remark, the main issues arising when considering a single-well logarithmic potential is that the positivity of the solution is not ensured at the discrete level, since the mobility degeneracy set $\{0;1\}$ does not coincide with the singularity set of the potential, \emph{i.e.}, $n = 1$. Therefore, the absence of cells represents an unstable equilibrium of the potential. In~\cite{agosti_cahn-hilliard-type_2017}, the authors design a finite element scheme which preserves positivity by the means of a discrete variational inequality, as also suggested in~\cite{barrett_finite_1999}. More recently, in~\cite{agosti_discontinuous_2019}, Agosti has presented a discontinuous Galerkin finite element discretization of the equation where, again, the positivity of the discrete solution is ensured thanks to a discrete variational inequality. 

\AP{
In this paper, we take advantage of the equivalent rewriting~\eqref{eq:CH-relax-equiv} of the RDCH system and use previous results obtained on the analysis of finite element schemes derived for degenerate parabolic equations. Namely, we use results obtained by Canc\`es and Guichard~\cite{cances} for a Control-Volume-Finite-Element (CVFE) used to simulate the anisotropic porous medium equation. In this work, the nonnegativity of the discrete solution is achieved by a suitable definition of the convective mobility that is discretized on control volumes. The compactness of time and space translates are achieved using discrete energy estimates, and the convergence analysis is obtained using the Frechet-Kolmogorov theorem. This work is the basis of a subsequent paper by Canc\`es \etal~\cite{cances:hal-01119210} in which a CVFE scheme for an anisotropic Keller-Segel system is analyzed. Therefore, using the fact that the RDCH system is close to the Keller-Segel model (\ie they are both composed of one degenerate parabolic equation and one elliptic equation), we combine the results of the previously presented works on CVFE schemes to perform our analysis. 
}
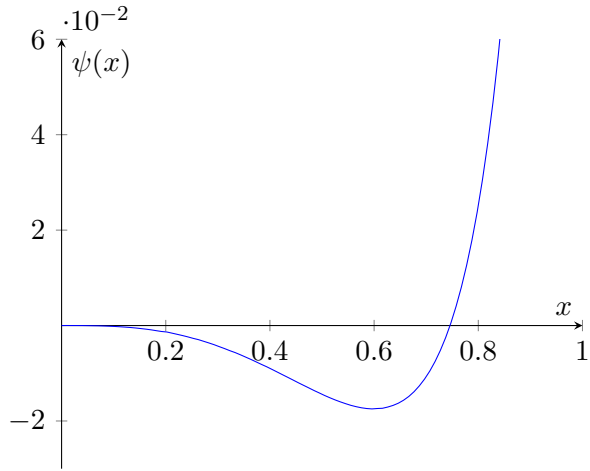
\begin{figure}
	\centering
\begin{tikzpicture}	
	\begin{axis}[
		axis lines = center,
		xlabel = $x$,
		ylabel = {$\psi(x)$},
		ymin=-0.03,
    	ymax=0.06,
		xmin=0,
    	xmax=1,
	]
	\addplot [
		domain=0:1, 
		samples=100, 
		color=blue,
	]
	{-(1-0.6)*ln((1-x))- x^3/3  -(1-0.6)*x^2/2 - (1-0.6)*x};

	\end{axis}
\end{tikzpicture}
\caption{Single-well potential of Lennard-Jones type as in \eqref{eq:potential} with $n^\star = 0.6$.}
	\label{fig:single-well}
\end{figure}

\paragraph{Contents of the paper.} 
\AP{
The aim of this paper is to present and analyze two finite element approximations of the relaxed Cahn-Hilliard equation with single-well potential~\eqref{eq:potential} and degenerate mobility~\eqref{eq:mobility} in dimensions $d=1,2,3$. More in detail, we propose two different time discretizations leading to one nonlinear scheme on which we can prove analytically some important properties and one efficient linear scheme that retrieve these important properties during numerical simulations (but we can not prove them analytically). 
However, for the nonlinear scheme, we prove analytically: (i) well-posedness of the numerical approximation; (ii) nonnegativity of discrete solutions ensured by adapting the upwind technique to the finite element approximation method; (iii) dissipation of a discrete energy; (iv) convergence of discrete solutions.\newline
}
\commentout{
In System~\eqref{eq:CH-relax}, $\psi_{+}$ and $\psi_{-}$ are, respectively, the convex and concave part of $\psi$, and defined for $0 \le n < 1$ as
\begin{equation}
\psi_+(n)  := -(1-n^\star)\log(1-n)- \f{n^3}{3}, \quad \text{and } \quad \psi_-(n) :=  -(1-n^\star)\f{n^2}{2} - (1-n^\star)n,
\label{eq:psi-plus-minus}
\end{equation}
where $\psi_+$ is convex whenever $n^\star\le 1-\left(\f2 3\right)^3$. The main difference in the definition of the potential for the RDCH model is that the concave part of the potential must be defined on all $\R$ and  bounded. Therefore, we assume that a proper extension of $\psi_-$ given in \eqref{eq:psi-plus-minus} exists and satisfies 
\begin{equation}
	\ov \psi_-\in C^2_b(\R,\R),
	\label{eq:assumption-psi_minus}
\end{equation}
where $C^2_b$ denotes the space of bounded continuous functions with bounded continuous first and second derivatives. We also assume that $\ov \psi^\prime_-$ remains linear.\newline
}

The main novelty of our work is to propose an alternative to the second-order splitting~\eqref{eq:CH-splitting} by replacing the chemical potential $\mu$ by its relaxed approximation $\vp$, solution to a second order elliptic equation with diffusivity $0 < \sigma \ll \gamma$. The relaxed system is based on the analysis performed in~\cite{poulain_relaxation_2019}, where the authors prove well-posedness of the system as well as the convergence, as $\sigma \to 0$, of weak solutions of~\eqref{eq:CH-relax} to the ones of the original Cahn-Hilliard equation~\eqref{eq:CH}. For the analysis that follows, it is worth noticing that System~\eqref{eq:CH-relax} admits the energy functional
\begin{equation}
\mathcal{E}_\sigma[n](t) := \int_{\Omega} \Bigg{\{}  \frac{\gamma}{2} \abs{\nabla \left( n - \frac{\sigma}{\gamma} \vp \right)}^2 + \frac{\sigma}{2 \gamma} \abs{\vp}^2 + \psi_{+}(n) + \psi_{-}\left(n - \frac{\sigma}{\gamma} \vp \right) \Bigg{\}} \, \text{d}x,
\label{eq:energy-CH}
\end{equation}
that, as proved in~\cite{poulain_relaxation_2019}, is decreasing in time, \emph{i.e.},
\begin{equation*}
\frac{\text{d} \cae_\sigma[n]}{\text{d}t} = - \int_\Omega b(n) \abs{\nabla \left(\vp + \psi_{+}^{\prime}(n) \right)}^2 \text{d}x \leq 0, \quad t > 0.
\end{equation*}
We also notice that the convex/concave splitting of $\psi$ is different from the one employed, \emph{e.g.}, in~\cite{agosti_cahn-hilliard-type_2017} and is motivated by the need to retrieve energy dissipation as well as by the fact that we can take advantage of the linearity of $\psi_{-}^{\prime}$ to achieve regularity results on $w$. Furthermore, we observe that the relaxed Cahn-Hilliard system bears some similarities with the Keller-Segel model with additional cross diffusion, proposed and analyzed in~\cite{bessemoulin_2014,carrillo_2012}.

\AP{
In this work, we use another definition of the continuous solutions for the equivalent system~\eqref{eq:CH-relax-equiv}. To do so, we define the non-decreasing continuous function $\eta: \R \to \R$ by 
\begin{equation}
    \label{eq:def-eta}
    \eta(r) = \int_0^r \sqrt{b(s)} \,\dd s, \quad \forall r \in \R.
\end{equation}
Similarly, we also define the non-decreasing continuous function $\zeta:\R\to \R$, using the properties of $b(s)\psi_+''(s)$ stated in~\eqref{eq:assbpsi}, such that
\begin{equation}
    \label{eq:def-zeta}
    \zeta(r) = \int_0^r \sqrt{b(s)\psi_+''(s)} \,\dd s, \quad \forall r \in \R.
\end{equation}
With the definition of these new functions, we define the solutions of the RDCH system~\eqref{eq:CH-relax-equiv}.
\begin{definition}[Weak solutions of RDCH]\label{def:solution-RDCH}
    The functions $(n,w)$ such that 
    \begin{equation}
        \begin{cases}
            0\le n\le 1, \text{a.e. in } \Omega_T,  \\
            \eta(n) \in L^2(0,T;H^1(\Omega)),\quad \zeta(n) \in L^2(0,T;H^1(\Omega)),\\
            w \in L^\infty(\Omega_T) \bigcap L^2(0,T;H^1(\Omega)),
        \end{cases}
    \end{equation}
    are the weak solutions of the relaxed-degenerate Cahn-Hilliard model~\eqref{eq:CH-relax-equiv} in the following sense: For all test functions $\chi_1,\chi_2 \in C^\infty_c(\Omega_T,R^+)$ with $\chi_1(T,\cdot)= \chi_2(T,\cdot)=0$, we have
    \begin{equation}
        \begin{aligned}
            -\int_\Omega n^0 \chi_1 - \int_{\Omega_T} n \f{\p \chi_1}{\p t} + \f{\gamma}{\sigma}\int_{\Omega_T} \sqrt{b(n)}\nabla\eta(n) \nabla \chi_1 + \int_{\Omega_T} \sqrt{b(n)\psi_{+}^{''}(n)}\nabla\zeta(n) \nabla \chi_1 \\= \f{\gamma}{\sigma} \int_{\Omega_T} b(n) \nabla w \nabla \chi_1,
        \end{aligned}
        \end{equation}
        \begin{equation}
            \sg\int_{\Omega_T} \nabla w\nabla\chi_2 +  \int_{\Omega_T} \left(\f\gamma\sigma\psi_-^\prime(w) + w\right) \chi_2 = \int_{\Omega_T} n \chi_2.
        \end{equation}
\end{definition}
}

\AP{
This paper is organized as follows. 
Section~\ref{sec:notation} presents the details of the Control-Volume-Finite-Element framework and the assumptions on the mesh.
In Section~\ref{sec:scheme}, we introduce a nonlinear semi-implicit finite element approximation of the relaxed Cahn-Hilliard equation~\eqref{eq:CH-relax}. The definition of the upwind mobility coefficient is given using a constant approximation of the mobility term on specific control volumes. 
Subsection~\ref{sec:positivity} and Subsection~\ref{sec:apriori} are dedicated to the presentation of the properties of this nonlinear scheme, such as the nonnegativity of the discrete solutions, mass conservation, and apriori estimates.
The existence of discrete solutions to this problem is given in Subsection~\ref{sec:existence}.  
In Subsection~\ref{sec:compact} and Subsection~\ref{sec:convergence}, we derive compactness estimates and prove the convergence of the discrete solutions to the weak solutions of the continuous relaxed Cahn-Hilliard model defined in Definition~\ref{def:solution-RDCH}. Then, Section \ref{sec:explicit} is devoted to the description of an efficient linear semi-implicit scheme. The existence and the nonnegativity of a unique solution are given for a suitable upwind approximation of the mobility coefficient.
Finally, in Section~\ref{sec:simulations}, we present numerical simulations using our linear semi-implicit scheme in one and two dimensions. These numerical simulations are in good agreement with previous numerical results obtained for the degenerate Cahn-Hilliard equation with single-well logarithmic potential. }
\section{Notations and assumptions}
\label{sec:notation}
We first set up the notations we will use in the numerical discretization and recall some well-known properties we employ in the analysis.
\begin{sloppypar}
\paragraph{Geometric and functional setting.} Let $\Omega \subset \mathbb{R}^d$ with $d=1,2,3$ be a polyhedral domain. We indicate the usual Lebesgue and Sobolev spaces by respectively $L^p(\Omega)$, $W^{m,p}(\Omega)$ with ${H^m(\Omega) := W^{m,2}(\Omega)}$, where $1 \le p \le +\infty$ and $m \in  \mathbb{N}$. We denote the corresponding norms by $||\cdot||_{m,p,\Omega}$, $||\cdot||_{m,\Omega}$ and semi-norms by $|\cdot|_{m,p,\Omega}$, $|\cdot|_{m,\Omega}$. The standard $L^2$ inner product will be denoted by $(\cdot,\cdot)_\Omega$ and the duality pairing between $(H^1(\Omega))'$ and $H^1(\Omega)$ by $<\cdot,\cdot>_\Omega$.\newline
Let $\mathcal{T}^h$, $h>0$ be a conformal mesh on the domain $\Omega$ which is defined by $N_\text{el}$ disjoint piecewise linear mesh elements, denoted by $K\in \mathcal{T}^h$, such that $\overline{\Omega} = \bigcup_{K\in \mathcal{T}^h} \overline{K}$. The elements are triangles for $d=2$ and tetrahedra for $d=3$. We let $h := \max_{K} h_K$ refers to the level of refinement of the mesh, where $h_K := \text{diam}(K)$ for $K \in \mathcal{T}^h$. We define by $\kappa_K$ the minimal perpendicular length of $K$ and $\kappa_h = \min_{K\in \mathcal{T}^h} \kappa_K$.
We assume that the mesh is quasi-uniform, \emph{i.e.}, it is shape-regular and there exists a constant $C>0$ such that 
\end{sloppypar}
\begin{equation*}
h_K \ge C h, \quad \forall K \in \mathcal{T}^h.
\label{eq:numquasi-uniform-num}
\end{equation*} 
Moreover, we assume that the mesh is acute, \textit{i.e.}, for $d = 2$ the angles of the triangles do not exceed $\frac{\pi}{2}$ and for $d=3$ the angle between two faces of the same tetrahedron do not exceed $\frac{\pi}{2}$.\newline
We define the set of nodal points $J_h=\{x_j\}_{j=1,\dots,N_h}$, where the number of nodes is $N_h := \text{card}(J_h)$, and we assume that each $x_j$ is a vertex of a simplex $K \in \mathcal{T}^h$. We denote by $\Lambda_i$ the set of nodes connected to the node $x_i$ by an edge and $G_h$ the maximal number of connected nodes, \ie $G_h = \max_{x_i\in J_h} \Lambda_i$. For two nodes $x_i, x_j$, if they are connected by an edge, we denote this latter $e_{ij}$.

We also define the barycentric dual mesh associated to $\mathcal{T}^h$. On each element $K\in \mathcal{T}^h$, the barycentric coordinates of an arbitrary point $x\in K$ are defined by the real numbers $\lambda_i$ with $i=1,\dots,n_K$ such that
\[
    \sum_{i=1}^{n_K} \lambda_i = 1,\quad \text{and} \quad x = \sum_{i=1}^{n_K}\lambda_i P_i,
\]
where $n_K$ is the number of nodes of the element $K$.
We define the barycentric subdomains associated to the vertex $P_i\in K_k$ (where $K_k$ refers to the $k$-th element of $\mathcal{T}^h$ and $k=1,\dots,N_\text{el}$), by
\[
    D_i^k := \bigcap_{\substack{j=1\\
                        j\neq i}}^{n_K} \{x; x\in K_k\text{ and } \lambda_{j}(x) \le \lambda_i(x) \}.
\]
Therefore, for each node of the mesh $\mathcal{T}^h$, we define the barycentric cell of the dual mesh
\[
    D_i := \bigcup_{k} \{ D_i^k; \, K_k \in \mathcal{T}^h \, \text{such that} \, x_i \in K_k\}.
\]

 We also add another subdivision of the mesh that relies on the definition of the barycentric cells and that will be useful to define the upwind approximation of the mobility. Using the barycentric coordinates,  we subdivide each element $K\in \mathcal{T}^h$ in $d+1$ subdomains, such that, for $i = 1,2,3$, $j=2,3$, and $i\neq j$, we have the subdomain
 \[
 \tilde D_{ij}^K = \{x\in K | \lambda_i, \lambda_j \ge \lambda_k,\quad k\neq i,j \}.
 \]
 In the following of the article, we use the terminology \textit{Diamond cell} for the two $\tilde D_{ij}^K $ that share the edge $e_{ij}$. 

 To illustrate what are these subdomains, we represent graphically what is $\tilde D_{ij}^T$ on Figure \ref{fig:subdomain} as well as what is a diamond cell for $d=2$.
 \begin{figure}
     \centering
     \includegraphics[width=0.45\linewidth]{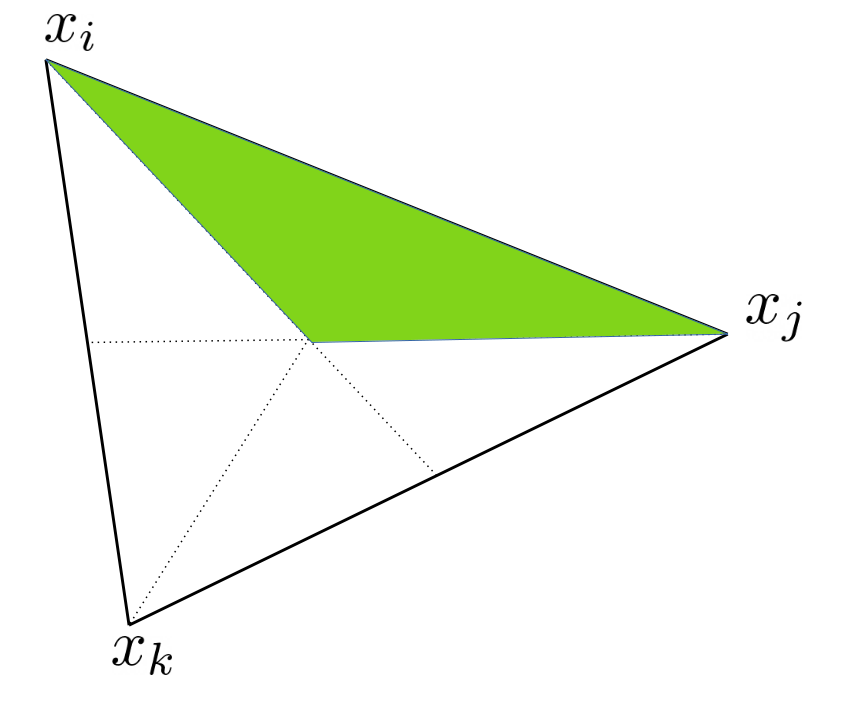}
     \includegraphics[width=0.4\linewidth]{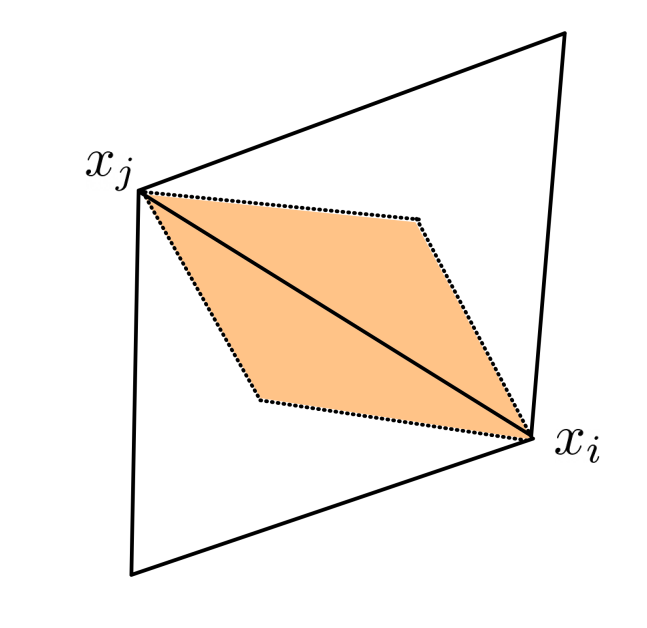}
     \caption{Illustration of the subdomain $\tilde D_{ij}^K$ (left, green) and of a diamond cell (right, orange) for $d=2$.}
     \label{fig:subdomain}
 \end{figure}

We introduce the set of piecewise linear functions $\chi_j \in C(\overline{\Omega})$ associated with the nodal point $x_j \in J_h$, that satisfies $\chi_j(x_i) = \delta_{ij}$, where $\delta_{ij}$ is the Kronecker's delta. 
We introduce the $\mathbb{P}^1$ conformal finite element space $V^h$ associated with $\mathcal{T}^h$, where $\mathbb{P}^1(K)$ denotes the space of polynomials of order $1$ on $K$: 
\begin{equation*} 
    V^h := \left\{ \chi \in C(\overline{\Omega}): \restr{\chi}K \in \mathbb{P}^1(K), \quad \forall K \in \mathcal{T}^h \right\} \subset H^1(\Omega).
\end{equation*}
Furthermore, we let $K^h$ be the subset containing the nonnegative elements of $V^h$, namely
\begin{equation*}
K^h:= \{ \chi \in V^h:  \chi \geq 0 \quad \text{in } \Omega \}.
\end{equation*}
We denote by $\pi^h \colon C(\overline{\Omega}) \to V^h$ the Lagrange interpolation operator corresponding to the discretized domain $\mathcal{T}^h$, defined as
\begin{equation*}
	\pi^h f(x) = \sum_{j =1}^{N_h} f(x_j) \, \bas_j(x), \quad f \in C(\overline{\Omega}).
\end{equation*}
We also use the lumped space $\hat V_h$ defined by
\[
    \hat V^h := \{\hat \bas : \text{ piecewise constant over barycentric domains, \ie } \hat \bas(x) = \hat \bas (x_i),\, \forall x \in D_i \}.
\]
Defining  $\hat \bas_i \in L^\infty(\Omega)$ the characteristic function of the barycentric domain $D_i$ associated with each node $x_i$ (for $i=1,\dots,N_h$) of the mesh, we easily see that $\{ \hat \bas_j\}_{j=1,\dots, N_h}$ forms a basis of $\hat V^h$. Adding the property $\hat \bas_i(x_j) = \delta_{ij}$, we see that the two basis $\{ \hat \bas_j\}_{j=1,\dots, N_h}$ and $\{\chi_j\}_{j=1,\dots,N_h}$ are associative, \ie $\bas(x_i)= \hat \bas(x_i)$ for all $x_i\in J_h$. 
Therefore, we also define the lumped operator  ${\hat \pi^h \colon C(\overline{\Omega}) \to \hat V^h}$ by
\[
    \hat \pi^h f(x) = \sum_{j =1}^{N_h} f(x_j) \, \hat \bas_j(x), \quad f \in C(\overline{\Omega}).
\]

On $C(\overline{\Omega})$ we define the approximate scalar product as
\begin{equation*}
(f,g)^h := \int_\Omega  \pi^h \left(f(x) \, g(x)\right) \, \text{d} x, \quad f, g \in C(\overline{\Omega}).
\end{equation*}
Furthermore, since $\forall f, g \in C(\overline{\Omega})$, we have
\[
    \begin{aligned}
    \int_\Omega \pi^h \left(f(x) \, g(x)\right) \, \text{d} x &= \sum_{K \in \mathcal{T}^h} \int_K \pi^h \left(f(x) \, g(x)\right) \, \text{d} x, \\
    &=  \f{1}{d+1}\sum_{K \in \mathcal{T}^h} \abs{K} \sum_{x_i \in K} f(x_i)\, g(x_i),\\
    &= \int_\Omega \hat \pi^h \left(f(x) \, g(x)\right) \, \text{d} x.
    \end{aligned}
\]
where $x_i \in K$ denotes the vertices of the element $K$.
We denote the corresponding discrete semi-norm as $\abs{\cdot}_h = \left[\lscal{\cdot}{\cdot} \right]^{\f12}$.

We denote by $P_h^0: L^2(\Omega)\to \fespace$ the $L^2$ projection operator and by $\hat P_h^0: L^2(\Omega)\to \fespace$ its lumped version, defined by
\begin{equation*}
\begin{aligned}
\scal{P_h^0 v}{\bas} &= \scal{v}{\bas} \quad \forall v\in L^2(\Omega)\text{ and } \forall \bas \in \fespace,\\
\lscal{\hat P_h^0 v}{\bas} &= \scal{v}{\bas} \quad \forall v\in L^2(\Omega)\text{ and } \forall \bas \in \fespace.
\end{aligned}
\end{equation*}

\commentout{
We define a lumped $H^1$-projection operator onto $V^h$, \ie $\hat P^1_h:H^1(\Omega) \to V^h$ such that $\forall v\in H^1(\Omega)$, we have
\begin{equation}
    \scal{\nabla \hat P_h^1(v)}{\nabla \chi} = \scal{ \nabla v}{\nabla \chi}, \quad \text{and}\quad \lscal{\hat P_h^1(v)}{\chi} = \lscal{v}{\chi},\quad  \chi \in V^h. 
    \label{eq:def-H1-proj}
\end{equation}
From this definition, we know that 
\[
    \hat P_h^1(v) \to v,\quad \text{in} \quad H^1(\Omega).  
\]
}
\commentout{
Furthermore,  we introduce the inverse Laplacian operator $\mathcal{G}: \mathcal{F}\to S$ as an application from $\mathcal{F} = \{v\in \left(H^1(\Omega)\right)' : <v,1> = 0 \}$ to $S = \{v\in H^1(\Omega) : (v,1) = 0 \}$, that satisfies
\begin{equation}
(\nabla \mathcal{G} v, \nabla \eta ) = <v,\eta> \quad \forall \eta \in H^1(\Omega).
\label{eq:numinverse-laplacian}
\end{equation}
The well-posedness of~\eqref{eq:numinverse-laplacian} follows immediately from the Lax-Milgram theorem and the Poincar\'e inequality. Therefore, a norm on $\mathcal{F}$ can be defined via
\[
\norm{v}_{\mathcal{F}}:= \abs{\mathcal{G}v}_1 \equiv \left<v,\mathcal{G}v \right>^{\f12}\quad \forall v\in \mathcal{F}.
\]
The discrete counterpart of $\mathcal{G}$ is denoted by $\hat{\mathcal{G}}^h : \mathcal{F}^h \to S^h$ and satisfies
\begin{equation*}
\scal{\nabla \hat{\mathcal{G}}^h v}{\nabla \chi} = \lscal{v}{\chi} \quad \forall \chi \in \fespace,
\end{equation*}
where $S^h:=\{ v^h\in \fespace : \scal{v^h}{1}=0\}$ and $\mathcal{F}^h:= \{ v\in C(\ov\Omega): \lscal{v}{1}=0\}$.
}
\paragraph{Inequalities.}
We summarize important inequalities that will be used later on for the analysis of the numerical schemes.  
\commentout{
We start by recalling the well-known Sobolev interpolation result. Letting $p\in [1,\infty]$, $m \ge 1$,
\begin{equation*}
r\in\begin{cases}
[p,\infty]\quad &\text{if    } m-\f dp>0,\\
[p,\infty) \quad &\text{if    } m-\f dp=0,\\
[p,-\f{d}{m-(d/p)}] & \text{otherwise,}
\end{cases}  
\end{equation*}
and $\mu = \f dm \left(\f 1p- \f 1r\right)$, there exists a constant $C = C(\Omega,p,r,m) > 0$ such that 
\begin{equation}
\norm{v}_{0,r} \le C \norm{v}_{0,p}^{1-\mu}\norm{v}^\mu_{m,p} \quad \forall v\in W^{m,p}(\Omega).
\label{eq:numsobolev-interpolation}
\end{equation}
}
We will use the following inequalities (see, \emph{e.g.},~\cite{quarteroni_numerical_1994}):
\begin{align}
    \abs{\chi}_{m,p_2} &\le C h^{-d\left(\f{1}{p_1}-\f{1}{p_2} \right)}\abs{\chi}_{m,p_1} \quad \forall \chi \in V^h, 1\le p_1\le p_2 \le  +\infty, m=0,1; \label{eq:ineq-comparaison-norm} \\
    \norm{\chi}_0^2& \le \lscal{\chi}{\chi} \le (d+2)\norm{\chi}^2_0 \quad \forall \chi \in V^h \label{eq:norm-lumped-scalar-prod}.
\end{align}
Concerning the interpolation operator, the following result holds~\cite{brenner_2008_fe}:
    \begin{equation}
        \lim_{h\to 0}\norm{v-\pi^h(v)}_{0,\infty}=0 \quad \forall v \in C(\ov \Omega),
        \label{eq:ineq-lumped-scalar-product2}
    \end{equation}
and we have \cite{vidar_1997_galerkin},
\begin{equation}
    \abs{\lscal{v}{\eta} -\scal{v}{\eta}} \le C h^2 \norm{\nabla v}_0 \norm{\nabla \eta}_0 ,\quad  v,\eta \in \fespace.
    \label{eq:conv-lump-scal}
\end{equation}
\commentout {
Furthermore, if $d=1$ (see for example \cite{vidar_1997_galerkin}),
    \begin{align}
        & \abs{v - \pi^h(v)}_{m,p} \le C h^{1-m}\abs{v}_{1,p} \quad \forall v \in W^{1,p}(\Omega),\quad m=0,1, \quad 1\le p <  +\infty; \label{eq:numineq-interp-Lp} \\
        &\norm{v - \pi^h(v)}_{L^\infty(\Omega)} + h \abs{v - \pi^h(v)}_{1,\infty} \le C h^2\abs{v}_{1,\infty} \quad \forall v \in H^1(\Omega) \label{eq:numineq-interp-Linf},\\
        &\abs{\lscal{v}{\eta} -\scal{v}{\eta}} \le C\left(\abs{v-\pi^h v}_0+h\abs{v}_0 \right) \norm{\eta}_1,\quad \text{for} \quad v\in C\left(\ov \Omega \right), \eta \in H^1(\Omega). \label{eq:numconv-lump-scal}
\end{align}
}
For the $L^2$ projection operator, the following inequality holds~\cite{brenner_2008_fe}
\begin{equation}
    \norm{v-P_h^0 v}_{0} + h\abs{v-P_h^0 v}_1 \le C h^{m} \abs{v}_{m} \quad v\in H^m(\Omega), \quad m=1,2.
    \label{eq:numineq-L2-proj}
\end{equation}
\commentout{
Finally, the discrete inverse laplacian operator satisfies 
\begin{equation}
    \lscal{v}{\chi} \equiv \scal{\nabla \hat{\mathcal{G}}^h v}{\nabla \chi} \le \abs{\hat{\mathcal{G}}^h v}_1 \abs{\chi}_1\quad \forall v \in \mathcal{F}^h,\chi \in \fespace.
    \label{eq:numineq-green-ope}
\end{equation}
}
\paragraph{Finite element matrices.} 
We define $M$ and $Q$ respectively the mass and stiffness matrix. $M_l$ is the lumped mass matrix, that is the diagonal matrix where each coefficient is the sum of the associated row of $M$
\begin{equation*}
	M_{ij} = \int_{\Omega} \chi_i  \chi_j \, \text{d}x, \quad \text{ for } i,j = 1, \dots, N_h,
\end{equation*} 
\begin{equation*}
	Q_{ij} = \int_{\Omega} \nabla \chi_i \cdot \nabla \chi_j \, \text{d}x, \quad \text{ for } i,j = 1, \dots, N_h,
\end{equation*} 
\begin{equation*}
	M_{l,ii} := \sum_{j=1}^{N_h} M_{ij}, 	 \quad \text{ for } i,j = 1, \dots, N_h.
\end{equation*} 

\AP{
We recall some important properties of the stiffness matrix. We start by 
\begin{equation}
    \sum_{j \neq i} Q_{ij} = Q_{ii},
    \label{eq:Qii-Qij}
\end{equation}
Furthermore, from the fact that the triangulation is of acute type, we know from \cite{fujii_1973_fe} that 
\begin{equation}
    Q_{ij} \le 0 \text{ for } i\neq j, \quad \text{and} \quad 0 < \f{\abs{Q_{ij}}}{M_{l,ii}}\le \f{Q_{ii}}{M_{l,ii}} \le \f{(d+1)}{\kappa_h^2}.
    \label{eq:estimate-ratio-Q-ML}
\end{equation}
}

\section{Nonlinear CVFE scheme}
\label{sec:scheme}
\subsection{Description of the scheme}

Given $N_T \in \mathbb{N}$, let $\Delta t := T / N_T$ be the constant time step size and $t^k := k \Delta t$, for ${k= 0, \dots, N_T - 1}$. We consider a partitioning of the time interval $[0,T] = \bigcup_{k=0}^{N_T-1}[t^k,t^{k+1}]$ with $0<T<+\infty$.
We approximate the continuous time derivative by $\f{\p n_h}{\p t} \approx \f{n^{k+1}_h - n^k_h}{\dt}$. 
Using the finite element approximations of $n$ and $w$, 
\begin{equation*}
    n_h^k(x) := \sum_{j =1}^{N_h} n^k_j \chi_j(x), \quad \text{and} \quad w_h^k(x) := \sum_{j =1}^{N_h} w^k_j \chi_j(x),
\end{equation*}
where $\{n^k_j\}_{j=1,\dots,N_h}$ and $\{w^k_j\}_{j=1,\dots,N_h}$ are the unknown degrees of freedom, we introduce the following finite element approximation of System~\eqref{eq:CH-relax}.

\textbf{Problem $P$: }For each $k = 0, \dots, N_T - 1$, find $\{n_h^{k+1},w_h^{k+1}\}$ in $K^h \times \fespace$ such that $\forall \chi_1,\chi_2 \in \fespace$
\begin{equation}
    \begin{aligned}
    \left(\frac{n_h^{k+1}-n_h^k}{\Delta t},\chi_1\right)^h  + \scal{\left(\f\gamma\sigma \hat B(n_h^{\kk}) +\widehat{(B\psi_+'')}(n^\kk_h) \right)\nabla n_h^\kk}{\nabla \chi_1}\\
    = \f\gamma\sigma\scal{\tilde B(n_h^{\kk}) \nabla  w_h^{k+1}}{\nabla \chi_1}  , \label{eq:discrete-n} 
    \end{aligned}
\end{equation}
\begin{equation}
    \sigma\left(\nabla w_h^{k+1},\nabla \chi_2\right) + \left(w_h^{k+1} + \f\gamma\sg \psi_-^\prime(w^k_h),\chi_2\right)^h  = \lscal{n^k_h}{\chi_2}. \label{eq:discrete-vp}
\end{equation}
where $\tilde B(\cdot)$ is the discrete upwind approximation of the continuous mobility $b(\cdot)$ for the convective term while $\hat B(\cdot)$ and $\widehat{B\psi_+''}(\cdot)$ are the constant approximations of the mobility and second derivative of the convex part of the potential multiplied by the mobility. These latter quantities are defined in Equations~\eqref{eq:mob-up-1D},~\eqref{eq:diff-b}, and~\eqref{eq:diff-psi} respectively.
The initial condition for the discrete problem is chosen such that for $n^0(x)\in H^1(\Omega)$, we have
\begin{equation}
    \begin{cases}
        n_h^0 &= \pi^h \left(n^0(x)\right),\quad \text{for}\quad d = 1,\\
        n^0_h &= \hat P_h^0 \left(n^0(x)\right),\quad\text{for}\quad d=2, \, 3, 
    \end{cases}
    \label{eq:init-discr1}
\end{equation}
and $w^0_h$ is the unique solution of
\begin{equation}
    \sg\scal{ \nabla w^0_h}{\nabla \chi} + \lscal{w_h^0 + \f\gamma\sg \psi_-^\prime(w^0_h) }{\bas} =  \lscal{ n^0_h}{ \bas} ,\quad \forall \chi \in \fespace.
    \label{eq:init-discr2}
\end{equation}
\commentout{
Similarly, we define the regularized problem using the smooth potential $\psi_\es$. 

\textbf{Problem $P_\es$: }for each $k = 0, \dots, N_T - 1$, find $\{n_\heps^{k+1},w_\heps^{k+1}\}$ in $K^h \times \fespace$ such that $\forall \chi \in \fespace$
\begin{subequations}
    \begin{numcases}{}
        \left(\frac{n_\heps^{k+1}-n_\heps^k}{\Delta t},\chi\right)^h  + \left(\tilde B(n_\heps^{k}) \nabla \left( \vp_\heps^{k+1} + \proj (\psi_{+,\es}^\prime(n_\heps^{k+1}))\right), \nabla \chi\right) = 0  , \label{eq:discrete-n-reg} \\
        \sigma\left(\nabla \vp_\heps^{k+1},\nabla \chi\right) + \left(\vp_\heps^{k+1},\chi\right)^h  = \gamma \left(\nabla n_\heps^{k+1/2},\nabla \chi\right) + \left(\ov \psi_-^\prime(n_\heps^{k}-\frac{\sigma}{\gamma}\vp_\heps^{k}),\chi\right)^h ,\label{eq:discrete-vp1-reg}
    \end{numcases}
\end{subequations}
}
To preserve the nonnegativity of the discrete solutions, we compute the mobility coefficient employing an implicit upwind method adapted to the finite element setting. The explanation on how to adapt the upwind method requires us to define the matrix formulation of the problem ~\eqref{eq:discrete-n}--\eqref{eq:discrete-vp}.  

\paragraph{Matrix form.} 
\label{sec:upwind}

For $k = 0, \dots, N_T -1$, let $\underline{n}^k$ and $\underline{w}^k$ be the vectors
\[
	\underline{n}^k := [ n_{1}^k, \dots, n^k_{N_h}  ]^{T}, \quad \underline{w}^k := [ w_{1}^k, \dots, w^k_{N_h}  ]^{T}.
\]
We can then rewrite System~\eqref{eq:discrete-n}-~\eqref{eq:discrete-vp} in its matrix form
\begin{align}
M_l \underline{n}^{k+1} + \Delta t R^\kk\underline{n}^{k+1} &= M_l \underline{n}^{k} + \Delta t U^\kk \underline{w}^{k+1}, \label{eq:discrete-n-matrix} \\
\left( \sigma Q + M_l \right) \underline{w}^{k+1} + \f\gamma\sg M_l \underline{\psi}_{-}^{\prime}  &= M_l \underline{n}^{k} , \label{eq:discrete-vp-matrix} 
\end{align}
where $\underline{\psi}_{-}^{\prime}$ is the vector containing the values at the nodal values
\begin{align*}
\left(\underline{\psi}_{-}^{\prime}\right)_i &=  \psi_-^\prime(w_i^k)\quad i=1,\dots,N_h.
\end{align*}

In the above matrix form, we have used the definitions of the finite element matrices associated with $U^\kk = U(n^\kk_h) = \scal{\tilde B(n^\kk_h)\nabla \cdot}{\nabla \cdot}$, such that
\begin{equation}
U_{ij}^\kk = \f\gamma\sigma \int_{\Omega} \tilde B(n^{\kk}_{h}) \nabla \chi_i \nabla \chi_j \, \text{d}x  = \f\gamma\sigma \int_{\Omega} \tilde B^{\kk}_{ij} \nabla \chi_i \cdot \nabla \chi_j \, \text{d}x, \quad \text{ for } i,j = 1, \dots, N_h,
\label{eq:definition-matrix-U}
\end{equation} 
and for $R^\kk = R(n^\kk_h) = \scal{ \f\gamma\sigma \hat B(n^\kk_h)+ \widehat{\left(B \psi^{''}_+ \right)}(n^\kk_h) \nabla \cdot}{\nabla \cdot}$, we have
\begin{equation}
    \begin{aligned}
R_{ij}^\kk &= \int_{\Omega} \f\gamma\sigma \hat B(n^\kk_h)+ \widehat{\left(B \psi^{''}_+ \right)}(n^\kk_h)\nabla \chi_i \cdot \nabla \chi_j \, \text{d}x  \\
&= \int_{\Omega} \f\gamma\sigma\hat B^{\kk}_{ij}+\widehat{\left(B\psi^{''}_{+}\right)}_{ij}^\kk  \nabla \chi_i \cdot\nabla \chi_j \, \text{d}x, \quad \text{ for } i,j = 1, \dots, N_h.
    \end{aligned}
\label{eq:definition-matrix-R}
\end{equation} 

From \eqref{eq:Qii-Qij} and \eqref{eq:estimate-ratio-Q-ML}, we can interpret the first equation of System~\eqref{eq:discrete-n}--\eqref{eq:discrete-vp} as using the finite volume method \ie, for every $x_i\in J_h$, we have
\begin{equation}
    \begin{aligned}
    M_{l,ii}n_i^\kk + \dt \sum_{x_j \in \Lambda_i}  \left(\f\gamma\sigma \hat B_{ij}^\kk + \widehat{\left(B \psi^{''}_+ \right)}^\kk_{ij}\right)& \abs{Q_{ij}} (n_i^\kk-n_j^\kk)   \\
    &= M_{l,ii}n_i^k + \f{\gamma\dt}{\sigma}  \sum_{x_j \in \Lambda_i}  \tilde B_{ij}^\kk \abs{Q_{ij}} (w_i^\kk-w_j^\kk).
    \end{aligned}
    \label{eq:equation-FV-form}
\end{equation}

Therefore, defining $F_{ij}^\kk = H_{ij}^\kk - G_{ij}^\kk$ with 
\[
\begin{aligned}
    H_{ij}^\kk &= H(n_i^\kk, n_j^\kk, (\delta n^\kk)_{ij}) =  \left(\f\gamma\sigma \hat B_{ij}^\kk + \widehat{\left(B \psi^{''}_+ \right)}^\kk_{ij}\right) Q_{ij} (-(\delta n^\kk)_{ij}),\\
    G_{ij}^\kk &= G(n_i^\kk, n_j^\kk, (\delta w^\kk)_{ij} )= \f\gamma\sigma\tilde B_{ij}^\kk Q_{ij} (-(\delta w^\kk)_{ij}),
    \end{aligned}
\]
where $(\delta w)_{ij}$ denotes the difference $(w_i^\kk-w_j^\kk)$.
Thus, the scheme reads
\begin{subequations}
    \begin{numcases}{}
        F_{ij}^\kk + F_{ji}^\kk = 0,\\
          M_{l,ii}n_i^\kk   = M_{l,ii}n_i^k - \dt  \sum_{j \neq i} F_{ij}^\kk, \quad \forall x_i\in J_h, \label{eq:discrete-n-FV} \\
        \sigma\left(\nabla w_h^{k+1},\nabla \chi\right) + \left(w_h^{k+1} + \f\gamma\sg \psi_-^\prime(w^k_h),\chi\right)^h  = \lscal{n^k_h}{\chi}. \label{eq:discrete-vp-FV}
    \end{numcases}
\end{subequations}

Let us now describe in detail how the coefficients $\tilde B^{\kk}_{ij}$, $\hat B^\kk_{ij}$ and $ \widehat{\left(B \psi^{''}_+ \right)}^\kk_{ij}$ are computed.   
\paragraph*{Constant approximation of mobility and second derivative of potential.}
We define the upwind mobility coefficient associated with the convective term by
\begin{equation}
	\tilde B^{\kk}_{ij} := \begin{cases} b_1(n^{\kk}_{i})b_2(n^{\kk}_{j}),\,&\text{if } \, w_j^{k+1} - w^{k+1}_i > 0,\\
	 b_1(n^{\kk}_{j})b_2(n^{\kk}_{i}),\,&\text{otherwise},
	\end{cases} \quad i, j = 1, \dots, N_h.
	\label{eq:mob-up-1D}
\end{equation}

The mobility function and second derivative of the convex part of the potential associated to the diffusion term are defined by 
\begin{equation}
	\hat B^{\kk}_{ij} := \max_{[n_i^\kk,n_j^\kk]} b(s)\quad i, j = 1, \dots, N_h,
	\label{eq:diff-b}
\end{equation}
and 
\begin{equation}
	\widehat{\left(B \psi^{''}_+ \right)}^\kk_{ij} := \max_{[n_i^\kk,n_j^\kk]} b(s) \psi_+''(s)\quad i, j = 1, \dots, N_h.
	\label{eq:diff-psi}
\end{equation}

The approximation~\eqref{eq:mob-up-1D} is similar to the one used in~\cite{almeida_energy_2019} for the one-dimensional finite volume discretization of the Keller-Segel system. Furthermore, our adaptation of the upwind method in the finite element context is also close in spirit to the one proposed by Baba and Tabata in~\cite{baba_1981_upwindfe}, where the authors also used barycentric coordinates to define the basis functions.\newline
We remark that the coefficients $\tilde B^\kk_{ij}$, $\hat B_{ij}^{\kk}$ and $\widehat{\left(B \psi^{''}_+ \right)}^\kk_{ij}$ are constant and uniquely defined along each edge of the mesh (\ie in the diamond cells). 
Yet, this method is well suited for a standard assembling procedure and, as a result, is simpler to implement in already existing finite element software since it requires only the adaptation of the calculation of a non-constant matrix. This method can also be adapted for the simulation of other advection-diffusion equations to preserve the nonnegativity of solutions.

From this approximation of the mobility we have the following properties concerning the convective flux.
\begin{proposition}[Properties of the convective flux]\label{prop:convective-flux}
For any $(a,b,c)\in \R^3$, we have: 
\begin{enumerate}
    \item $G(\cdot, b, c)$ is non-decreasing, and $G(a, \cdot, c)$ is non-increasing;
    \item $G(a, b, c) = - G(b, a, -c)$;
    \item $G(a,a,c) = b(a)  (-c)$;
    \item there is a constant $C>0$ such that $\abs{G(a,b,c)}\le C(\abs{a}+\abs{b})\abs{c}$;
    \item there is a modulus of continuity $\omega:\R^+\to  \R^+$ such that for all $(a',b',c')\in \R^3$
    \[
        \abs{G(a,b,c) - G(a',b',c')} \le \abs{c}\omega\left(\abs{a-a'}+\abs{b-b'} \right).
    \]
\end{enumerate}
\end{proposition}

\subsection{Discrete maximum principle, non-negativity, and mass conservation}
\label{sec:positivity}

\begin{proposition}[Non-negativity of $n_h^{k+1}$ and upper bound]
    The numerical scheme~\eqref{eq:discrete-n}--\eqref{eq:discrete-vp} preserves the non-negativity of the solution, i.e. for all $k=0,\dots,N_T-1$, we have $0\le n_h^{k+1} \le 1$.
    \label{prop:non-neg-n}
\end{proposition}
\begin{proof}
    \textit{Step 1: Non-negativity.} We prove the non-negativity by induction. We assume that $\forall x_i \in J_h$, $ 0\le n_i^k  \le 1$, and $n_i^\kk = \min_{x_j\in J_h}(n_j^\kk)$. Then, we multiply Equation~\eqref{eq:discrete-n-FV} by $-(n_i^\kk)^- = -\max(0, -n_i^\kk )$ to obtain
    \[
    \begin{aligned}
    -\sum_{j=1}^{N_h}\scal{\hat \bas_j}{\hat \bas_i}\left(\f{n^\kk_i-n_i^k}{\dt}\right)(n^\kk_i)^- = &\sum_{j\neq i}  \left(\f\gamma\sigma \hat B_{ij}^\kk + \widehat{\left(B \psi^{''}_+ \right)}^\kk_{ij}\right)Q_{ij}(n_j^\kk-n_i^\kk)(n^\kk_i)^-     \\
    &- \f\gamma\sigma\sum_{j\neq i} \tilde B_{ij}^\kk Q_{ij} (w^\kk_j-w^\kk_i)(n_i^\kk)^-.
    \end{aligned}
    \]
    From the non-negativity of $\hat B_{ij}^\kk$, and $\widehat{\left(B \psi^{''}_+ \right)}^\kk_{ij}$ as well as \eqref{eq:estimate-ratio-Q-ML}, we know that 
    \[
        \sum_{j\neq i}  \left(\f\gamma\sigma \hat B_{ij}^\kk + \widehat{\left(B \psi^{''}_+ \right)}^\kk_{ij}\right)Q_{ij}(n_j^\kk-n_i^\kk)(n^\kk_i)^- \le 0.
    \]
    Then, if $w_j^\kk - w_i^\kk > 0$, $\tilde B_{ij}^\kk = b_1(n_i^\kk)b_2(n_j^\kk)$, and from the extension of $b_1(s)$ by zero for $s<0$, we have 
    \[
    - \f\gamma\sigma\sum_{j\neq i} \tilde B_{ij}^\kk Q_{ij} (w^\kk_j-w^\kk_i)(n_i^\kk)^- \le 0.
    \]
    Altogether, we obtain
    \[
        -\sum_{j=1}^{N_h}\scal{\hat \bas_j}{\hat \bas_i}\left(\f{n^\kk_i-n_i^k}{\dt}\right)(n^\kk_i)^- \le 0,
    \]
    and, hence, $(n^\kk_i)^- = 0$ which implies $n^\kk_i \ge 0$ and $n_h^\kk \ge 0$.
    
    \textit{Step 2: Upper bound.} To prove the upper bound we repeat the previous argument but this time we assume $\forall x_i \in J_h$, $ 0\le n_i^k  \le 1$, and $n_i^\kk = \max_{x_j\in J_h}(n_j^\kk)$. We multiply Equation~\eqref{eq:discrete-n-FV} by $(n_i^\kk-1)^+ = \max(0, n_i^\kk-1 )$ to obtain  
    \[
    \begin{aligned}
     \sum_{j=1}^{N_h}\scal{\hat \bas_j}{\hat \bas_i}\left(\f{n^\kk_i-n_i^k}{\dt}\right)(n_i^\kk-1)^+ = &-\sum_{j\neq i}  \left(\f\gamma\sigma \hat B_{ij}^\kk + \widehat{\left(B \psi^{''}_+ \right)}^\kk_{ij}\right)Q_{ij}(n_j^\kk-n_i^\kk)(n_i^\kk-1)^+    \\
    &+ \f\gamma\sigma\sum_{j\neq i} \tilde B_{ij}^\kk Q_{ij} (w^\kk_j-w^\kk_i)(n_i^\kk-1)^+.
    \end{aligned}
    \]
    Since $n_i^\kk = \max_{x_j\in J_h}(n_j^\kk)$ and from the non-negativity of $\hat B_{ij}^\kk$ and $ \widehat{\left(B \psi^{''}_+ \right)}^\kk_{ij}$ as well as \eqref{eq:estimate-ratio-Q-ML}, we have
    \[
    -\sum_{j\neq i}  \left(\f\gamma\sigma \hat B_{ij}^\kk + \widehat{\left(B \psi^{''}_+ \right)}^\kk_{ij}\right)Q_{ij}(n_j^\kk-n_i^\kk)(n_i^\kk-1)^+  \le 0.
    \]
    Then, if $w_j^\kk - w_i^\kk \le 0$, we have $\tilde B_{ij}^\kk = b_1(n_j^\kk)b_2(n_i^\kk)$, and from the extension of $b_2(s)$ by zero for $s>1$, we obtain
    \[
        \f\gamma\sigma\sum_{j\neq i} \tilde B_{ij}^\kk Q_{ij} (w^\kk_j-w^\kk_i)(n_i^\kk-1)^+ \le 0.
    \]
    Altogether, we have 
    \[
        \sum_{j=1}^{N_h}\scal{\hat \bas_j}{\hat \bas_i}\left(\f{n^\kk_i-n_i^k}{\dt}\right)(n_i^\kk-1)^+ \le 0,
    \]
    which implies $(n_i^\kk-1)^+=0$, and, hence, $n_i^\kk\le 1$. Therefore, we obtain the result. 
\end{proof}

The previous result allows to show that $w_h^\kk$ is confined in a threshold $[-C,C]$ with $C>0$ finite. 
\begin{proposition}
    The numerical scheme~\eqref{eq:discrete-n}--\eqref{eq:discrete-vp} preserves a lower and upper bound for $w^\kk_h$, i.e. for all $k=0,\dots,N_T-1$, we have $-C\le w_h^{k+1} \le C$, for $C>0$ finite.
    \label{prop:confin-w}
\end{proposition}
\begin{proof}
    The result is found using Equation~\eqref{eq:discrete-vp}, Proposition~\ref{prop:non-neg-n}, and the boundedness of ${\psi}_-^\prime(s)$ for all $s\in \R$ (given by~\eqref{eq:psi-minus}).
\end{proof}

\begin{proposition}[Conservation of mass]
    The finite element numerical scheme~\eqref{eq:discrete-n}--\eqref{eq:discrete-vp} preserves the initial mass, i.e. for all $k=0,\dots,N_T-1$, we have
    \[
        \int_\Omega n^0_h\,\dd x = \int_\Omega n^\kk_h\,\dd x.   
    \]
\end{proposition}
\begin{proof}
    To prove mass conservation, we use the identity, for each $x_i\in J_h$
        \begin{equation}
        \sum_{\substack{ x_j \in \Lambda_i}} \abs{U_{ij}^\kk} = U_{ii}^\kk, \text{ and }  \sum_{\substack{ x_j \in \Lambda_i}} \abs{R_{ij}^\kk} = R_{ii}^\kk.
        \label{eq:prop-A} 
    \end{equation}
    Summing over the nodes in Equation~\eqref{eq:discrete-n}, we obtain
    \[
        \begin{aligned}
            \sum_{i=1}^{N_h}\sum_{j=1}^{N_h}\scal{\hat \bas_j}{\hat \bas_i}\left(n^{k+1}_h-n^k_h\right)(x_i) 
            = &-\dt \sum_{i=1}^{N_h}\sum_{j=1}^{N_h} U_{ij}^\kk (-w_h^{k+1})(x_j) \\
            &-\dt \sum_{i=1}^{N_h}\sum_{j=1}^{N_h} R_{ij}^\kk n_h^\kk(x_j) . 
        \end{aligned}
    \]
    Using the symmetry of the matrices $U^\kk$ and $R^\kk$, the property \eqref{eq:prop-A}, we obtain 
    \[
        \sum_{i=1}^{N_h}\sum_{j=1}^{N_h}\scal{\hat\bas_j}{\hat \bas_i}\left(n^{k+1}_h-n^k_h\right)(x_i) = 0,
    \]
    which implies mass conservation.
\end{proof}

\subsection{Energy stability and a priori estimates}
\label{sec:apriori}
The finite element numerical scheme~\eqref{eq:discrete-n}--\eqref{eq:discrete-vp} preserves the dissipation of the energy at the discrete level. 

\begin{proposition}[Energy estimate for $w_h^\kk$]
\label{prop:estimate-H1-w}
    System~\eqref{eq:discrete-n}--\eqref{eq:discrete-vp} admits the following apriori estimate
    \[
        \sum_{k=0}^{N_T-1}\dt \abs{w^\kk_h}^2_1 = \sum_{k=0}^{N_T-1}\dt \sum_{x_i\in J_h}\sum_{x_j\in \Lambda_i} \abs{Q_{ij}} (w_i^\kk - w_j^\kk)^2 \le C(T, w^0).
    \]
\end{proposition}
\begin{proof}
    We use $\chi = \dt w^\kk_h$ in Equation~\eqref{eq:discrete-vp}, sum from $k=0 \to N_T-1$, use the boundedness of $\norm{n^\kk_h}_{\infty}$ and $\norm{w^\kk_h}_{\infty}$ given by the propositions~\ref{prop:non-neg-n} and \ref{prop:confin-w}, as well as the boundedness of $\psi_-^\prime$ to obtain the result.
\end{proof}

\begin{proposition}[Energy estimate on $n_h^\kk$ and dissipation]
    \label{prop:energy}
    System~\eqref{eq:discrete-n}--\eqref{eq:discrete-vp} admits the apriori estimate
    \begin{equation}
		\begin{aligned}
		\f12\norm{n_h^\kk}^2_0 + \sum_{k=0}^{N_T-1} \dt \sum_{x_i\in J_h}\sum_{x_j\in\Lambda_i} \left( C_2 \hat B^\kk_{ij} + \widehat{\left(B \psi^{''}_+ \right)}^\kk_{ij} \right)\abs{Q_{ij}}(n_i^\kk-n_j^\kk)^2 \le C(T,w^0_h,n^0_h),
		\end{aligned}
		\label{eq:energy}
     \end{equation}
     where $C_2$ is a small positive finite constant.
\end{proposition}
\begin{proof}
    Multiplying Equation~\eqref{eq:equation-FV-form} by $n_i^\kk$ and summing over the nodes $x_i\in J_h$ leads to 
    \[
    \begin{aligned}
        \lscal{n_h^\kk-n_h^k}{n^\kk_h} + \dt\sum_{x_i\in J_h}\sum_{x_j\in\Lambda_i}&  \left(\f\gamma\sigma \hat B_{ij}^\kk + \widehat{\left(B \psi^{''}_+ \right)}^\kk_{ij}\right)\abs{Q_{ij}}(n_i^\kk-n_j^\kk)^2 \\
        &=  \f\gamma\sigma \dt\sum_{x_i\in J_h}\sum_{x_j\in\Lambda_i}\tilde B_{ij}^\kk \abs{Q_{ij}}(w_i^\kk-w_j^\kk)(n_i^\kk-n_j^\kk).
        \end{aligned}
    \]
    The term on the right hand side can be bounded using Young's Inequality, and the boundedness of the mobility to obtain, for any $\kappa>0$,   
    \[
    \begin{aligned}
        \dt\sum_{x_i\in J_h}\sum_{x_j\in\Lambda_i}& \f\gamma\sigma\tilde B_{ij}^\kk \abs{Q_{ij}}(w_i^\kk-w_j^\kk)(n_i^\kk-n_j^\kk) \\
        &\le \dt\sum_{x_i\in J_h}\sum_{x_j\in\Lambda_i} \f{\kappa \gamma^2}{2\sigma^2}  \abs{Q_{ij}}(w_i^\kk-w_j^\kk)^2 + \f1{2\kappa} \tilde B_{ij}^\kk \abs{Q_{ij}}(n_i^\kk-n_j^\kk)^2 .
    \end{aligned}
    \]
    Therefore, it leads to 
    \[
    \begin{aligned}
        \lscal{n_h^\kk-n_h^k}{n^\kk_h} + \dt\sum_{x_i\in J_h}\sum_{x_j\in\Lambda_i}& \left(\f\gamma\sigma\hat B^\kk_{ij}- \f1{2\kappa} \tilde B^\kk_{ij}+ \widehat{\left(B\psi^{''}_{+}\right)}^{\kk}_{ij}\right)\abs{Q_{ij}}(n_i^\kk-n_j^\kk)^2 \\
        &\le \dt\sum_{x_i\in J_h}\sum_{x_j\in\Lambda_i} \f{\kappa \gamma^2}{2\sigma^2}  \abs{Q_{ij}}(w_i^\kk-w_j^\kk)^2.
        \end{aligned}
    \]
    Therefore, from the boundedness of both $\hat B_{ij}^\kk$ and $\tilde B^\kk_{ij}$, as well as the fact that $\hat B_{ij}^\kk \ge\tilde B^\kk_{ij}$ by definition, and since $\kappa$ can be chosen arbitrarily ,we know that it exists a finite constant $C_2>0$ such that $0 < C_2 \hat B^\kk_{ij} \le \f\gamma\sigma\hat B^\kk_{ij}- \f1{2\kappa} \tilde B^\kk_{ij}$ (we recall that the term $\f\gamma \sigma$ does not induce any difficulty since $\sigma< \gamma$).
    Using the property $(a-b)a \ge \f12 (a^2-b^2)$, Proposition~\ref{prop:estimate-H1-w} as well as the property~\eqref{eq:norm-lumped-scalar-prod}, and summing for $k=0\to N_T-1$, we obtain the result.
\end{proof}

\begin{corollary}
    \label{cor:estimate}
    The following inequalities hold 
    \begin{equation}
        \sum_{k=0}^{N_T-1}\dt\abs{\eta^{k+1}_h}_1^2  \le C(T,n^0,w^0), \text{ and } \sum_{k=0}^{N_T-1}\dt\abs{\zeta^{k+1}_h}_1^2 \le C(T,n^0,w^0). 
        \label{eq:apriori-eta}
    \end{equation}
\end{corollary}
\begin{proof}
    From Proposition~\ref{prop:energy} and the non-negativity of $\widehat{\left(B \psi^{''}_+ \right)}^\kk_{ij}$, we know that 
    \[
         \sum_{k=0}^{N_T-1} \dt \sum_{x_i \in J_h} \sum_{x_j\in \Lambda_i} \hat B^\kk_{ij}\abs{Q_{ij}}(n_i^\kk-n_j^\kk)^2 \le C.
    \]
    Furthermore, from the definition~\eqref{eq:diff-b} and the definition of $\eta$ provided by~\eqref{eq:def-eta}, we have 
    \[
        \begin{aligned}
        \sum_{k=0}^{N_T-1} \dt\sum_{x_i \in J_h} \sum_{x_j\in \Lambda_i} \hat B^\kk_{ij}\abs{Q_{ij}}&(n_i^\kk-n_j^\kk)^2 \\
        &\ge \sum_{k=0}^{N_T-1} \dt \sum_{x_i \in J_h} \sum_{x_j\in \Lambda_i} \abs{Q_{ij}}(\eta_i^\kk-\eta_j^\kk)^2 = \sum_{k=0}^{N_T-1} \dt \abs{\eta_h^\kk}_1^2.
        \end{aligned}
    \]
    
    The same is obtained for $\sum_{k=0}^{N_T-1} \dt\abs{\zeta^{k+1}_h}^2_1$ using similar arguments. 
\end{proof}

\subsection{Existence of discrete solution}
\label{sec:existence}
We prove well-posedness of our problem.
\begin{theorem}[Well-posedness of the problem] \label{th:existence}
    \begin{sloppypar}
    Let $d\le 3$, and the spatio-temporal mesh satisfies the assumptions of Section~\ref{sec:notation}.
    Then, System~\eqref{eq:discrete-n}--\eqref{eq:discrete-vp} with an initial condition satisfying~\eqref{eq:init-discr1}--~\eqref{eq:init-discr2}, has a solution ${\{n^{k+1}_{h}, w_{h}^{k+1}\} \in K^h\times \fespace}$ with
    \begin{equation}
        0\le n^{k+1}_h \le 1,\quad \text{ in } \Omega.
        \label{eq:bound-sol}
    \end{equation}
    \end{sloppypar}
\end{theorem}
\begin{proof}
    The proof of the existence of a solution for the discrete problem relies on the use of Brouwer's fixed point theorem. 
    A necessary step before applying the theorem is to change the unknown of our problem, we define $m^{k+1}_h = n^{k+1}_h -\alpha$, where $\alpha = \f{1}{\abs{\Omega}} \int_\Omega n^0 \, \dd x$. Therefore, the discrete problem~\eqref{eq:discrete-n}--\eqref{eq:discrete-vp} is equivalent to 
    \[
         F(\undel m^{k+1})=\undel m^{\kk},  
    \]
    where $F:S \to S$ is an application defined on the space $S$ with
    \[
        S = \{\undel m \,|\,  M_l \undel m \cdot (1,\dots,1) = 0 \quad \text{and}\quad -\alpha \le {\undel m} \le 1-\alpha\}.
    \]
    The first constraint in the definition of the space $S$ reflects the conservation of the initial mass while the second comes from Inequality~\eqref{eq:bound-sol}. As a result, $S$ is a convex and compact subspace of $\R^{N_h}$.  
    Using the matrix formulation~\eqref{eq:discrete-n-matrix}--\eqref{eq:discrete-vp-matrix} for the problem $P$, the application $F$ is defined by 
    \[
        F(\undel m^{k+1}) = \dt  M_l^{-1}\left[ U^\kk \left(\sg Q + M_l \right)^{-1} \left(M_l \undel n^k -\f\gamma\sigma M_l \undel \psi_-^\prime \right) \right] -\dt R^\kk \undel m^\kk + \undel n^k - \alpha.
    \]
    In the previous definition of $F$, we precise that the matrix $U^\kk$ associated to the nonlinear convection term is given by 
    \[
        U_{ij}^\kk = \f\gamma\sigma \int_\Omega  \tilde B(m^\kk_h +\alpha) \nabla \chi_i \cdot \nabla \chi_j \,\dd x, 
    \]    
    and for the diffusion term 
    \[
        R_{ij}^\kk = \int_{\Omega} \left(\f\gamma\sigma\hat B(m^{\kk}_h+\alpha) +\widehat{\left(B\psi^{''}_{+}\right)}(m^\kk_h+\alpha)\right)  \nabla \chi_i \cdot \nabla \chi_j \, \text{d}x , 
    \]    
    To check if $F$ is continuous, we compute for $\undel m_1,\undel m_2 \in S$,
    $ \norm{F(\underline{m}_1) - F(\underline{m}_2)}$, 
    and obtain 
    \[
		\begin{aligned}
        \norm{F(\underline{m}_1) - F(\underline{m}_2)} \le \dt \norm{  M_l^{-1}\left[ \left(U(m_{h,1}+\alpha) - U(m_{h,2}+\alpha)\right) \left(\sg Q + M_l \right)^{-1} \left(M_l \undel n^k -\f\gamma\sigma M_l \undel \psi_-^\prime \right) \right]   }\\
        + \dt \norm{R(m_{h,1}+\alpha)\undel m_1 - R(m_{h,2}+\alpha)\undel m_2  }.	
		\end{aligned}
	\] 
For the first term on the right-hand side of the previous inequality, we have  
\[
       \norm{M_l^{-1}}\le C,\quad \norm{\left(\sg Q + M_l \right)^{-1}}\le C,\quad \norm{M_l \undel n^k -\f\gamma\sigma M_l \undel \psi_-^\prime}\le C,
\]
from the properties of the mesh, the upper bound presented by Varah~\cite{varah_lower_1975} for the inverse of M-matrices, and our assumptions on $n^k_h$ and $w^k_h$ as well as the function $ \psi_-(\cdot)$.  
Lastly, from the definitions of the matrix $U(\cdot)$ and of the convective flux $G$ (see Definition~\ref{prop:convective-flux}), we have 
\[
    \norm{U(m_{h,1}+\alpha) - U(m_{h,2}+\alpha)} \le C \norm{\underline{m}_{1} - \undel{m}_{2}} .
\]  
The same applies using the properties of the matrix $R(\cdot)$ to obtain
\[
    \norm{R(m_{h,1}+\alpha)\undel m_1 - R(m_{h,2}+\alpha)\undel m_2 } \le C \norm{\undel{m}_{1} - \undel{m}_{2}}.
\]
 Altogether, using the properties of the standard finite element matrices and the fact that the mobility $\tilde B(m^k_\heps + \alpha)$ is bounded, we obtain 
	\[
		\norm{F(\underline{m}_1) - F(\underline{w}_2)}	\le C(\Delta t, h)\norm{\underline{m}_1-\underline{m}_2},
	\]
	which proves that the mapping $F$ is Lipschitz continuous.
    Therefore, applying Brouwer's fixed point theorem, we know that it exists a solution $\undel m^{k+1} \in S$ such that $F(\undel m^\kk) = \undel m^\kk$. Therefore, it exists a $m^\kk_h \in \fespace$ that gives the existence of a pair $\{n^\kk_h,w^\kk_h\}\in \fespace\times \fespace$ solution of the Problem $P$.

\end{proof}

\subsection{Compactness estimates}
\label{sec:compact}
We now derive bounds for the time and space translates of the discrete solutions. This results follow the lines of the works~\cite{cances,cances:hal-01119210}. We first define the space and time discrete spaces $\fespace_\dt$ and $\hat V^h_\dt$. These sets are composed of piecewise constant functions in time with values in $\fespace$ and $\hat V^h$ respectively. To study the convergence as $h,\dt \to 0$, we use an index $m$ such that, as $m\to \infty$, $h_m,\dt_m\to 0$. Therefore, we study the convergence of sequences of functions in the spaces  $V^{h_m}_{\dt_m}$ and $\hat V^{h_m}_{\dt_m}$.

We define by $\eta_{h_m, \dt_m}$, $\zeta_{h_m, \dt_m}$ and $w_{h_m,\dt_m}$ the piecewise affine in space and piecewise constant in time approximation of the functions $\eta$, $\zeta$ and $w$. We also denote by $\hat \eta_{h_m, \dt_m}$, $\hat \zeta_{h_m, \dt_m}$ and $\hat w_{h_m,\dt_m}$, the piecewise constant in space and time corresponding approximations. For each node $x_i \in J_h$ and time $t=t^{k}$, we have the notation $\eta_{h_m, \dt_m}(x_i,t^k) = \hat \eta_{h_m, \dt_m}(x_i,t^k) = \eta^k_i$. The associated vector containing all the value for all nodes is denoted by $\undel \eta^k$. 

\paragraph{Time translate estimates}
We start by estimates on the time translates. 
We denote by $Q_{T-\tau} = \Omega \times (0,T-\tau)$, for all $\tau \in(0,T)$.
We have the following result 
\begin{lemma}[Time translate estimates]
    There are constants $C_{\eta,t}$, $C_{\zeta,t}$ and $C_{w,t}$ independent of $h$ and $\tau$ such that the following inequalities hold 
    \begin{align}
        \int_{Q_{t-\tau}} \abs{\hat\eta_{h_m, \dt_m}(x,t+\tau) - \hat\eta_{h_m, \dt_m}(x,t)}^2\,\dd x \,\dd t \le C_{\eta,t}(\tau+\dt), \label{eq:time-translate-nu}\\
        \int_{Q_{t-\tau}} \abs{\hat\zeta_{h_m, \dt_m}(x,t+\tau) - \hat\zeta_{h_m, \dt_m}(x,t)}^2\,\dd x \,\dd t \le C_{\zeta,t}(\tau+\dt), \label{eq:time-translate-zeta}\\
        \int_{Q_{t-\tau}} \abs{\hat w_{h_m, \dt_m}(x,t+\tau) - \hat w_{h_m, \dt_m}(x,t)}^2\,\dd x \,\dd t \le C_{w,t}(\tau+\dt). \label{eq:time-translate-w}
    \end{align}
    \label{lem:time-translate}
\end{lemma}
\begin{proof}
    The proof is similar to the proof of Lemma 4.3 in~\cite{cances}. We give here the details for the sake of clarity since the equation under study is different. 

    We start by defining the quantity
    \[
        A_m(t) = \int_\Omega \abs{\hat\eta_{h_m, \dt_m}(x,t+\tau) - \hat\eta_{h_m, \dt_m}(x,t)}^2\,\dd x, \quad \forall t\in (0,T-\tau).  
    \]
    Since the functions $\hat\eta_{h_m, \dt_m}$ are piecewise constant in time, we define $\upsilon(t)$, for $t\in (0,T]$, the unique integer such that $t^{\upsilon(t)}< t \le t^{\upsilon(t)+1}$. 
    Therefore, we write $\forall t\in (0,T-\tau)$,
    \[
        \begin{aligned}
        A_m(t) &= \left(\undel \eta^{\upsilon(t+\tau)+1} - \undel \eta^{\upsilon(t)+1} \right)^T M_l \left(\undel \eta^{\upsilon(t+\tau)+1} - \undel \eta^{\upsilon(t)+1} \right), \\
        &= \sum_{i=1}^{N_h} \left(\eta_i^{\upsilon(t+\tau)+1} -  \eta_i^{\upsilon(t)+1} \right)^2 M_{l,ii}.  
        \end{aligned}
    \]
    However, from definition~\eqref{eq:def-eta}, we know that it exists a constant $C$ such that
    \[
        \left(\eta_i^{\upsilon(t+\tau)+1} -  \eta_i^{\upsilon(t)+1} \right)^2 \le C \left(n_i^{\upsilon(t+\tau)+1} - n_i^{\upsilon(t)+1} \right)\left(\eta_i^{\upsilon(t+\tau)+1} -  \eta_i^{\upsilon(t)+1} \right).
    \] 
    Then, using the definition of the integer $\upsilon(\cdot)$, we find
    \[
        C \left(n_i^{\upsilon(t+\tau)+1} - n_i^{\upsilon(t)+1} \right)\left(\eta_i^{\upsilon(t+\tau)+1} -  \eta_i^{\upsilon(t)+1} \right) = C \sum_{k=\upsilon(t)+1}^{\upsilon(t+\tau)}\left(n_i^{k+1} - n_i^{k} \right)\left(\eta_i^{\upsilon(t+\tau)+1} -  \eta_i^{\upsilon(t)+1} \right).
    \]
    From the first equation of the scheme~\eqref{eq:discrete-n}, we obtain 
    \[
        \begin{aligned}
        A_m(t) \le - C \sum_{k=\upsilon(t)+1}^{\upsilon(t+\tau)} \dt_m \sum_{i=0}^{N_h} \sum_{x_j \in \Lambda_i}  \left(\f\gamma\sigma\hat B_{ij}^\kk+ \widehat{\left(B\psi_{+}^{''}\right)}^\kk_{ij} \right) \abs{Q_{ij}} (n_i^\kk-n_j^\kk)\\
        \times \left(\left(\eta_i^{\upsilon(t+\tau)+1} -\eta_j^{\upsilon(t+\tau)+1}\right)  -  \left(\eta_i^{\upsilon(t)+1} - \eta_j^{\upsilon(t)+1}\right)  \right)\\
        + C \sum_{k=\upsilon(t)+1}^{\upsilon(t+\tau)} \f{\gamma\dt_m}{\sigma}  \sum_{i=0}^{N_h} \sum_{x_j \in \Lambda_i}  \tilde B_{ij}^\kk \abs{Q_{ij}} (w_i^\kk-w_j^\kk)\\
        \times \left(\left(\eta_i^{\upsilon(t+\tau)+1} -\eta_j^{\upsilon(t+\tau)+1}\right)  -  \left(\eta_i^{\upsilon(t)+1} - \eta_j^{\upsilon(t)+1}\right)  \right).
        \end{aligned}
    \]
    Using Young's inequality, we have 
    \[
        A_m(t) \le  C\left(A_{1,m}(t) + A_{2,m}(t) + A_{3,m}(t) + A_{4,m}(t)\right),
    \] 
    where 
    \[
        \begin{aligned}
        A_{1,m}(t) &=  \f12\sum_{k=\upsilon(t)+1}^{\upsilon(t+\tau)} \dt_m \sum_{i=0}^{N_h} \sum_{x_j \in \Lambda_i}  \left(\f\gamma\sigma\hat B_{ij}^\kk+ \widehat{\left(B\psi_{+}^{''}\right)}^\kk_{ij} \right) \abs{Q_{ij}} (n_i^\kk-n_j^\kk)^2,\\
        A_{2,m}(t) &= \left(\f\gamma\sigma\norm{b}_\infty + \f12\norm{b\psi_+^{''}}_{\infty}\right) \sum_{k=\upsilon(t)+1}^{\upsilon(t+\tau)} \dt_m \sum_{i=0}^{N_h} \sum_{x_j \in \Lambda_i} \abs{Q_{ij}} \left(\eta_i^{\upsilon(t+\tau)+1} -\eta_j^{\upsilon(t+\tau)+1}\right)^2,\\
        A_{3,m}(t) &= \left(\f\gamma\sigma\norm{b}_\infty + \f12\norm{b\psi_+^{''}}_{\infty}\right) \sum_{k=\upsilon(t)+1}^{\upsilon(t+\tau)} \dt_m \sum_{i=0}^{N_h}  \sum_{x_j \in \Lambda_i} \abs{Q_{ij}} \left(\eta_i^{\upsilon(t)+1} - \eta_j^{\upsilon(t)+1}\right)^2,\\
        A_{4,m}(t) &= \f12\sum_{k=\upsilon(t)+1}^{\upsilon(t+\tau)} \dt_m \sum_{i=0}^{N_h}  \sum_{x_j \in \Lambda_i}\f\gamma\sigma \tilde B_{ij}^\kk \abs{Q_{ij}} (w_i^\kk-w_j^\kk)^2. 
        \end{aligned}
    \]
    To handle the first sum in each of these quantities, we introduce some additional notations. We define $\rho(k,t,\tau)$ the characteristic function such that 
    \[
        \rho(k,t,\tau) = \begin{cases}
            1,\text{ if } t < k\dt_m\le t+\tau ,\\
            0,\text{ otherwise}.
        \end{cases}  
    \]
    Hence, we have 
    \[
        \int_0^{T-\tau}  \rho(k,t,\tau)\,\dd t = \int_{t^k-\tau}^{t^k}\,\dd t = \tau,\quad \text{and}\quad  \sum_{k=\upsilon(t)+1}^{\upsilon(t+\tau)} \dt_m =\sum_{k: t<t^k\le t+\tau} t^{\kk}-t^k \le \tau + \dt_m.
     \]
    
    Furthermore, for any family of real non-negative values $(a^k)_{k\in\{0,\dots,N_T\}}$, we have 
    \[
        \int_0^{T-\tau} \sum_{k=\upsilon(t)+1}^{\upsilon(t+\tau)} \dt_m a^\kk\,\dd t = \int_0^{T-\tau} \sum_{k=0}^{N_T-1} \dt_m a^\kk \rho(k,t,\tau)\, \dd t = \tau \sum_{k=0}^{N_T-1} \dt_m a^\kk .
    \]

    From Proposition~\ref{prop:energy} and Proposition~\ref{prop:estimate-H1-w}, we find
    \[
        \int_0^{T-\tau} A_{1,m}(t)\,\dd t \le C(\tau + \dt_m),\text{ and }  \int_0^{T-\tau} A_{4,m}(t)\,\dd t \le C\tau \le C(\tau + \dt_m).
    \] 
    Then, using arguments similar to Proposition 9.3 in~\cite{Eymard-2003-petroleum}, we have  for any family of real non-negative values $(a^k)_{k\in\{0,\dots,N_T\}}$,
    \[
        \int_0^{T-\tau} \sum_{k=\upsilon(t)+1}^{\upsilon(t+\tau)} \dt_m a^{\upsilon(t+\tau)+1}\,\dd t \le \left(\tau + \dt_m\right)\sum_{k=0}^{N_T-1}(\dt_m a^\kk).
    \]
    Using the previous argument and Corollary~\ref{cor:estimate}, we obtain 
    \[
        A_{2,m} + A_{3,m} \le C(\tau + \dt).
    \]
    This achieves the proof of Inequality~\eqref{eq:time-translate-nu}. The proofs of Inequalities~\eqref{eq:time-translate-zeta} and~\eqref{eq:time-translate-w} are very similar and we do not repeat them. 
\end{proof}
\begin{remark}
    From the previous Lemma, we can easily give the time translate estimates over $\R^{d+1}$. Indeed, extending by zero the functions $\hat \eta_{h_m, \dt_m}$, $\hat \zeta_{h_m, \dt_m}$ and $\hat w_{h_m,\dt_m}$ outside of $\Omega \times (0,T)$, we obtain 
    \[
        \begin{cases}
            \int_{\R^{d+1}} \abs{\hat\eta_{h_m, \dt_m}(x,t+\tau) - \hat\eta_{h_m, \dt_m}(x,t)}^2\,\dd x \,\dd t \le C_{\eta,t}(\tau+\dt),\\
            \int_{\R^{d+1}} \abs{\hat\zeta_{h_m, \dt_m}(x,t+\tau) - \hat\zeta_{h_m, \dt_m}(x,t)}^2\,\dd x \,\dd t \le C_{\zeta,t}(\tau+\dt),\\
            \int_{\R^{d+1}}  \abs{\hat w_{h_m, \dt_m}(x,t+\tau) - \hat w_{h_m, \dt_m}(x,t)}^2\,\dd x \,\dd t \le C_{w,t}(\tau+\dt). 
        \end{cases}  
    \]
    To obtain this result, we use 
    \[
        \begin{aligned}
        \int_{\R^{d+1}} \abs{\hat\eta_{h_m, \dt_m}(x,t+\tau) - \hat\eta_{h_m, \dt_m}(x,t)}^2\,\dd x \,\dd t =  \int_{Q_{t-\tau}} \abs{\hat\eta_{h_m, \dt_m}(x,t+\tau) - \hat\eta_{h_m, \dt_m}(x,t)}^2\,\dd x \,\dd t \\
        + \int_{T-\tau}^T \int_\Omega \abs{\hat\eta_{h_m, \dt_m}(x,t)}^2\dd x \,\dd t ,
        \end{aligned}
    \]
    with Lemma~\ref{lem:time-translate} and the $L^\infty$-bounds on $\hat \eta_{h_m, \dt_m}$, $\hat \zeta_{h_m, \dt_m}$, and $\hat w_{h_m,\dt_m}$.
\end{remark}
\paragraph{Space translate estimates}
We now turn to the space translate estimates.
\begin{lemma}[Space translate estimates] \label{lem:space-translate}
    There are three constants $C_{\eta,s}$, $C_{\zeta,s}$ and $C_{w,s}$ independent of $m$ and $y$ such that 
    \begin{align}
        \int_0^T\int_{R^d} \abs{\hat\eta_{h_m, \dt_m}(x+y,t) - \hat\eta_{h_m, \dt_m}(x,t)}\,\dd x \,\dd t \le C_{\eta,s}(\abs{y}+h_m), \label{eq:space-translate-1}\\
        \int_0^T\int_{R^d} \abs{\hat\zeta_{h_m, \dt_m}(x+y,t) - \hat\zeta_{h_m, \dt_m}(x,t)}\,\dd x \,\dd t \le C_{\zeta,s}(\abs{y}+h_m), \label{eq:space-translate-3}\\
        \int_0^T\int_{R^d} \abs{\hat w_{h_m, \dt_m}(x+y,t) - \hat w_{h_m, \dt_m}(x,t)}\,\dd x \,\dd t \le C_{w,s}(\abs{y}+h_m) \label{eq:space-translate-2}.
    \end{align}
\end{lemma}
\begin{proof}
    The proof of this result follows Lemmas 4.1 and 4.2 in~\cite{cances}.
    For the sake of clarity, we here present the main steps. 

    First, using Corollary~\ref{cor:estimate}, we know that $\norm{\nabla \eta_{h_m,\dt_m}}_{\left(L^2(\Omega\times(0,T))\right)^d}$ is bounded. Since the time and space domain considered is of finite measure, H\"older inequality indicates that $\norm{\nabla \eta_{h_m,\dt_m}}_{\left(L^1(\Omega\times(0,T))\right)^d}$ is also bounded. Hence, since $\eta_{h_m,\dt_m}$ is bounded as well, its extension by zeros outside $\Omega\times(0,T)$ lies in $L^\infty \bigcap BV(\R^{d+1})$.
    Therefore, we have 
    \[
        \int_0^T\int_{\R^d}\abs{\eta_{h_m,\dt_m}(t,x+y)- \eta_{h_m,\dt_m}(t,x)}\, \dd x\, \dd t \le C\abs{y},   
    \]
    from which, using the triangular inequality and Remark~\ref{rem:hat-et-pas-hat} stated below, we find Inequality~\eqref{eq:space-translate-1}. The same is found for Inequality~\eqref{eq:space-translate-3} and Inequality~\eqref{eq:space-translate-2} following the same arguments.
\end{proof}

\begin{remark}\label{rem:hat-et-pas-hat}
    One of the important tool that we did not present in the previous proof is 
    \begin{equation}
        \int_{\Omega_T}\abs{\eta_{h_m,\dt_m}(t,x+y)-\hat \eta_{h_m,\dt_m}(t,x+y)}\,\dd x\,\dd t\le C h_m,
        \label{eq:hat-et-pas-hat}
    \end{equation}
    found from the use of Lemma A.2 in~\cite{cances}.
\end{remark}

\subsection{Convergence analysis}
\label{sec:convergence}
This section is organized as follow, we first apply Fr\'echet-Kolmogorov theorem to show strong convergence in $L^1(\Omega_T)$, then we show that the limit is a solution of the continuous RDCH system. 

\begin{lemma}[Strong convergence in $L^1$]
    \label{lem:conv}
    \begin{sloppypar}
    As $m\to \infty$, we can extract subsequences of $(\hat \eta_{h_m,\dt_m})_{m\ge 1}$, $(\hat \zeta_{h_m,\dt_m})_{m\ge 1}$ and $(\hat w_{h_m,\dt_m})_{m\ge 1}$  converging in $L^{1}(\Omega_T)$ to the limit functions $\eta(n)\in L^2(0,T;H^1(\Omega))$, $\zeta(n)\in L^2(0,T;H^1(\Omega))$ and $w \in L^2(0,T;H^1(\Omega))$ such that 
    \end{sloppypar}
    \begin{equation}
        \hat\eta_{h_m,\dt_m}, \hat \zeta_{h_m,\dt_m}, \hat w_{h_m,\dt_m} \to \eta(n), \zeta(n), w,\quad \text{strongly in}\quad L^1(\Omega_T),
        \label{eq:strong-conv}
    \end{equation}
    \begin{equation}
        \eta_{h_m,\dt_m},  \zeta_{h_m,\dt_m}, w_{h_m,\dt_m} \rightharpoonup \eta(n), \zeta(n), w,\quad \text{weakly in}\quad L^2(0,T;H^1(\Omega)).
        \label{eq:weak-conv}
    \end{equation}
    Furthermore, we have 
    \begin{equation}
        \hat n_{h_m,\dt_m} \to n,\quad \text{a.e. in}\quad \Omega_T,\text{ and strongly in } L^p(\Omega_T),\quad \text{for} \quad p<+\infty.
        \label{eq:convergence-n}
    \end{equation}
\end{lemma}
\begin{proof}
    Again, the proof is similar to the proof of Lemma 4.5 in~\cite{cances}. Indeed, from Lemmas~\ref{lem:time-translate} and~\ref{lem:space-translate}, and the boundedness provided by Corollary~\ref{cor:estimate}, we know that the sequence $(\hat \eta_{h_m,\dt_m})_{m\ge0}$ satisfies the necessary assumptions to use Fr\'echet-Kolmogorov theorem. 
    Therefore, we know that  $(\hat \eta_{h_m,\dt_m})_{m\ge0}$ is relatively compact in $L^1(\Omega_T)$, which implies~\eqref{eq:strong-conv} (the result for the sequences $(\hat \zeta_{h_m,\dt_m})_{m\ge0}$ and $(\hat v_{h_m,\dt_m})_{m\ge0}$ follows similar arguments). Furthermore, the limit is in $L^2(0,T;H^1(\Omega))$ since from the use of Corollary~\ref{cor:estimate}, we know that $( \eta_{h_m,\dt_m})_{m\ge0}$ also converges weakly in $L^2(0,T;H^1(\Omega))$ to a limit. 
    Since we know from Inequality~\eqref{eq:hat-et-pas-hat}, that the sequences $( \eta_{h_m,\dt_m})_{m\ge0}$ and $(\hat \eta_{h_m,\dt_m})_{m\ge0}$ have the same limit and $(\hat \eta_{h_m,\dt_m})_{m\ge0}$ already converges to a limit in $L^1(\Omega_T)$, uniqueness of the limit gives the result. The same applies for the quantities $\hat \zeta_{h_m,\dt_m}, \hat w_{h_m,\dt_m}$ and $ \zeta(n)_{h_m,\dt_m}, w_{h_m,\dt_m}$ such that we obtain the convergences~\eqref{eq:strong-conv}--\eqref{eq:weak-conv}.
    
    To prove the convergence~\eqref{eq:convergence-n}, we start by stating that the inverse $\eta^{-1}$ of the continuous function $\eta$ exists and is continuous. Furthermore, we have that $\hat n_{h_m,\dt_m}$ is bounded in $L^\infty$ from Proposition~\ref{prop:non-neg-n}. Therefore, applying the dominated convergence theorem to $\hat n_{h_m,\dt_m} = \eta^{-1}\left(\hat \eta_{h_m,\dt_m} \right)$, we arrive to the convergence~\eqref{eq:convergence-n} in which the limit is defined as the unique function $n(t,x) = \eta^{-1}(\eta)$ ($\eta$ being the limit function in~\eqref{eq:strong-conv}). 
\end{proof}

\begin{theorem}[Limit system] \label{th:conv}
    The limit of $(\hat n_{h_m,\dt_m}, \hat w_{h_m,\dt_m}, \hat \eta_{h_m,\dt_m}, \hat \zeta_{h_m,\dt_m})$ denoted $(n,w,\eta,\zeta)$ is solution of the RDCH system in the sense of Definition~\ref{def:solution-RDCH}.
\end{theorem}
\begin{proof}
    The proof of this theorem is very close to Section 5 in~\cite{cances:hal-01119210}. For the sake of clarity, our proof can be found in Appendix~\ref{app:proof}.
\end{proof}

\section{Linearized semi-implicit numerical scheme}
\label{sec:explicit}
To restrain the computational time of the simulation of the RDCH model within reasonable bounds, we propose a linearized semi-implicit version of the numerical scheme. \AP{We apply the CVFE framework to the original formulation of the RDCH problem this time.} The problem now reads:

For each ${k = 0, \dots, N_T -1}$, find~$\{n_h^{k+1},\varphi_h^{k+1}\}$ in $K^h \times V^h$ such that
\begin{subequations}
\begin{align}
        \left(\frac{n_h^{k+1}-n_h^k}{\Delta t},\chi_1\right)^h  + \left(b(n_h^{k})\psi^\seconde_+(n_h^k) \nabla n^\kk_h, \nabla \chi_1\right) =  -\left(\tilde B(n_h^{k}) \nabla  \vp_h^{k+1},\nabla \chi_1 \right) \, , \quad \forall \chi_1 \in V^h, \label{eq:numdiscrete-n-expl} \\
        \sigma\left(\nabla \vp_h^{k+1},\nabla \chi_2\right) + \left(\vp_h^{k+1},\chi_2\right)^h  = \gamma \left(\nabla n_h^{k},\nabla \chi_2\right) + \left( \psi_-^\prime(n_h^{k}-\frac{\sigma}{\gamma}\vp_h^{k}),\chi_2\right)^h \, , \quad \forall \chi_2 \in V^h.
        \label{eq:numdiscrete-vp-expl}
\end{align}
\end{subequations}
We define the following finite elements matrices
\begin{equation}
    U_{ij}^k = \int_{\Omega} \tilde B_{ij}^{k} \nabla \chi_i \cdot \nabla \chi_j \, \text{d}x, \quad \text{ for } i,j = 1, \dots, N_h,
    \label{eq:numdefinition-matrix-U-expl}
\end{equation}
and 
\begin{equation}
    L_{ij}^k = \int_{\Omega} b(n_{h,\epsilon}^{k}) \psi^\seconde_+(n^k_h)\nabla \chi_i \cdot \nabla \chi_j \, \text{d}x, \quad \text{ for } i,j = 1, \dots, N_h.
    \label{eq:numdefinition-matrix-D-expl}
\end{equation}
We write the matrix form of Equation~\eqref{eq:numdiscrete-n-expl}
\[
      \left(M_l + \Delta t L^k\right)\undel n^\kk = -\Delta t U^k \undel \vp^\kk + M_l \undel n^k,
\]
and since $U^k$ has zero row sum, we can rewrite the previous equation for each node $x_i\in J_h$,
\begin{equation}
	M_{l,ii} n^{k+1}_i =  M_{l,ii} n_i^k - \Delta t \sum_{x_j \in \Lambda_i } \left[ L_{ij}^k (n^{k+1}_j-n^{k+1}_i) + U_{ij}^k (\vp^{k+1}_j-\vp^{k+1}_i)\right],
	\label{eq:numform-fv}
\end{equation}
where $\Lambda_i$ is the set of nodes connected to the node $x_i$ by an edge.
In the definition of \eqref{eq:numdefinition-matrix-U-expl} we compute the mobility coefficient in function of the direction of $\nabla \vp^\kk_h$. As for the nonlinear case, the mobility coefficient is given by 
\begin{equation*}
	\tilde B_{ij}^{k}=
\begin{cases}
    n_i^k(1-n_j^k)^2 , \quad &\text{if} \quad \vp^{k+1}_i - \vp^{k+1}_j > 0,\\
    n^{k}_j(1-n^{k}_{i})^2, \quad &\text{otherwise}.
\end{cases}
\end{equation*}
Even though we cannot redo the same analysis as for the nonlinear scheme, we can establish the existence and the nonnegativity of discrete solutions of ~\eqref{eq:numdiscrete-n-expl} and~\eqref{eq:numdiscrete-vp-expl}.

\begin{theorem}[Well-posedness of linear upwind scheme]
	Let $\Omega \subset \R^d$, $d=1,2,3$, and assume that $\mathcal{T}^h$ is a quasi-uniform acute mesh of $\Omega$, and the condition
	\begin{equation}
		\f{(d+1)\, G_h\, \dt}{\kappa_h^2 } \max_{\substack{x_i \in J_h\\ x_j \in \Lambda_i}}\left( \vp_j^\kk-\vp_i^\kk\right) < 1,
		\label{eq:stab-condition}
	\end{equation}
    \begin{sloppypar}
	(where $\Lambda_i$ is the set of nodes connected to the node $x_i$ by an edge) is satisfied.
	Then, the linear finite element scheme ~\eqref{eq:numdiscrete-n-expl}--\eqref{eq:numdiscrete-vp-expl}  with initial condition $n^0_h \in K^h$ admits a unique solution ${\{n^{k+1}_h, \vp^\kk_h\}\in K^h\times V^h}$ satisfying
    \end{sloppypar}
	\begin{equation*}
	0 \leq	n_h^{k+1} < 1.
	\end{equation*}
\end{theorem}

\begin{proof}
\textit{Step 1. Existence of a unique solution. }
Assuming that $\{n^k_h,\vp_h^k\} \in K^h\times V^h$, from the Lax-Milgram theorem, it exists a unique solution $\vp^\kk_h \in V^h$ of Equation~\eqref{eq:numdiscrete-vp-expl} and Equation~\eqref{eq:numdiscrete-n-expl} admits a unique solution $n^\kk_h\in V^h$. 
Therefore, it exists a unique pair of discrete solutions $\{n^\kk_h,\vp^\kk_h\} \in V^h\times V^h$ for the system ~\eqref{eq:numdiscrete-n-expl}--\eqref{eq:numdiscrete-vp-expl}.
Next, we need to prove that $n^\kk_h$ is nonnegative and bounded from above by $1$.

\textit{Step 2. Nonnegativity and upper bound on ${n^\kk_h}$ for $d=1,2,3$. }
First, from the fact that $(M_l+\dt L^k)$ is a M-matrix, we know that its inverse is non-negative, i.e. $(M_l+\dt L^k)^{-1} \ge 0$.
Therefore, to preserve the non-negativity of $n^\kk_h$, we need that 
\[
    M_l \underline{n}^k -\dt U^k \underline{\vp}^{k+1}\ge 0.
\]
For every node $x_i\in J_h$, the previous condition reads
\[
    \abs{D_i} n_i^k -\dt \sum_{x_j\in \Lambda_i} B_{ij}^k Q_{ij} \left(\vp_i^{k+1}-\vp_j^{\kk}\right) \ge 0,
\]
where $\Lambda_i$ is the set of nodes connected to the node $x_i$ by an edge. 
From the fact that the mesh is acute, we know that $Q_{ij}$ is negative. Therefore, using the definition of the mobility coefficient \eqref{eq:mob-up-1D}, we need to focus on the case $\vp_j^{k+1}-\vp_i^{\kk} < 0$. In that situation, we have 
\[
    n_i^k - \f{\dt}{\abs{D_i}} \sum_{x_j\in \Lambda_i} n_i^k\left(1-n_j^k\right)^2 Q_{ij} \left(\vp_j^\kk-\vp_i^\kk \right)\ge 0.
\]

However, from~\eqref{eq:estimate-ratio-Q-ML}, we find the following condition to ensure the non-negativity of $n^\kk_h$
\[
     \f{(d+1)\, G_h\, \dt}{\kappa_h^2 } \max_{\substack{x_i \in J_h\\ x_j \in \Lambda_i}}\left( \vp_j^\kk-\vp_i^\kk\right) \le 1.
\]

Then, we need to prove that for every node $x_i\in J_h$ we have $n^\kk_i<1$. We use the upper bound presented by Varah~\cite{varah_lower_1975} for the inverse of M-matrices, and write 
\[
    \norm{ \left(\f{M_l}{\dt} +L^k\right)^{-1} }_\infty \le \f{\dt}{M_{l,ii}}.
\]
Therefore, to retrieve the upper bound on the discrete solution, the condition
\[
     n_i^k - \f{\dt}{\abs{D_i}} \sum_{x_j\in \Lambda_i} n_j^k\left(1-n_i^k\right)^2 Q_{ij} \left(\vp_j^\kk-\vp_i^\kk \right) <1,
\]
has to be satisfied. Note in the previous equation that we have considered the case $\vp_j^\kk-\vp_i^\kk >0$ since in the other case the bound will be satisfied trivially. 
Then, subtracting $n_i^k$ from both sides of the previous inequality, we obtain 
\[
      - \f{\dt}{\abs{D_i}} \sum_{x_j\in \Lambda_i} n_j^k\left(1-n_i^k\right) Q_{ij} \left(\vp_j^k-\vp_i^k \right) <1,
\]
and we retrieve the same condition as before with a strict inequality. 

Altogether, we proved the existence of a unique solution $\{n^\kk_h,\vp^k_h \}\in K^h\times S^h$ for the system ~\eqref{eq:numdiscrete-n-expl}--\eqref{eq:numdiscrete-vp-expl} with $0\le n^\kk_h < 1$ if the stability condition \eqref{eq:stab-condition} is satisfied.
\end{proof}

\section{Numerical simulations}
\label{sec:simulations}

\AP{
In this section, we use the previously presented linear scheme~\eqref{eq:numdiscrete-n-expl}--\eqref{eq:numdiscrete-vp-expl} for the RDCH system.	
The numerical simulations are performed using the MATLAB software. At each time step the matrices $U^k$ and $L^k$ are reassembled. The linear system is solved using the function \textit{linsolve} of the MATLAB software. This function uses the LU factorization. }

Even though we are presenting numerical results obtained using the linear scheme~\eqref{eq:numdiscrete-n-expl}--\eqref{eq:numdiscrete-vp-expl}, the evolution of the energy during the simulations is given from the computation of the discrete formulation of the continuous energy
\begin{equation*}
	E(n_h^\kk,\vp_h^\kk) := \int_\Omega \frac{\gamma}{2} \abs{\nabla\left( n_h^\kk  - \frac{\sigma}{\gamma} \vp_h^\kk\right)}^2 + \frac{\sigma}{2\gamma}|\vp_h^\kk|^2 + \psi_{+} (n_h^\kk)+ \psi_{-}\left(n_h^\kk -\frac{\sigma}{\gamma}\vp_h^\kk\right) \, \dd x .
	\label{eq:energy-simu}
 \end{equation*}
First of all, we present test cases in one and two dimensions to validate our method. The physical properties of the solutions such as the shape of the aggregates, the energy decay, the mass preservation and the non-negativity of the solution are the key characteristics we need to observe to validate our method. A comparison with previous results from the literature is also of main importance. The reference used for this study is the work of Agosti \textit{et al.} \cite{agosti_cahn-hilliard-type_2017}. The analysis of the long-time behavior of the solutions of the RDCH equation \cite{poulain_relaxation_2019} gives us some insights about what we should observe at the end of the simulations. The solutions should evolve to steady-states that are minimizers of the energy functional. Depending on the initial mass, three regions of the cell density should appear. The first being the region of absence of cells, the second the continuous interface linking the bottom and the top of the aggregates. If the initial mass is sufficiently large, the third expected region is a plateau of the cell density close to $n=n^\star$. 
\AP{This evolution is related to the clustering of tumor cells in \textit{in-vitro} biological experiments (see \eg \cite{bubba,agosti_self-organised_nodate}).}

The study of the effect of the regularization on the numerical scheme is the purpose of the last subsection.

\par

\subsection{Numerical results: test cases}
\subsubsection*{1D test cases}
Table \ref{tab:parameter1D} summarizes the parameters used for the one dimensional test cases. The initial cell density is a uniform distributed random perturbation around the values $n^0$.
\begin{table}
	\centering
	\caption{Parameters of the 1D test case}
	\begin{tabular}{|c|c|}
		\hline
		&Parameters  \\
		\hline
		$\gamma$& $(0.014)^2$ \\
		$\Delta t$& $0.1 \gamma$\\
		$h$ & $0.01$  \\
		$n^0$ & $\{0.05,0.3,0.36\}$\\
		$n^\star$& $0.6$ \\
		$\sg$&  $5.10^{-5}$\\
		\hline
	\end{tabular}
	\label{tab:parameter1D}
\end{table}
Figures ~\ref{fig:sol3steps} show the evolution in time of the solutions $n_h$ for the three different initial masses. 

\begin{figure}[h!!]
	\centering
	\includegraphics[scale = 0.31]{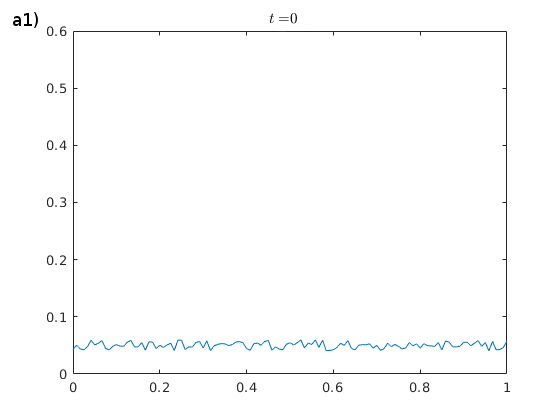}
	\includegraphics[scale = 0.33]{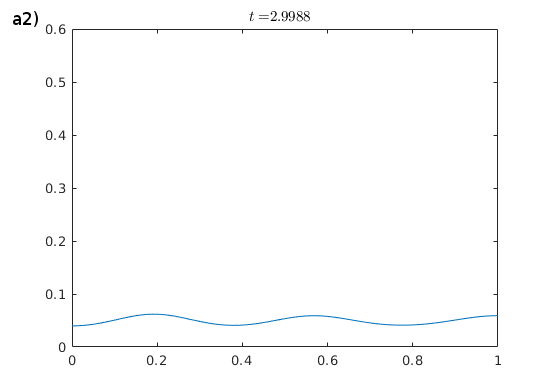}
	\includegraphics[scale = 0.33]{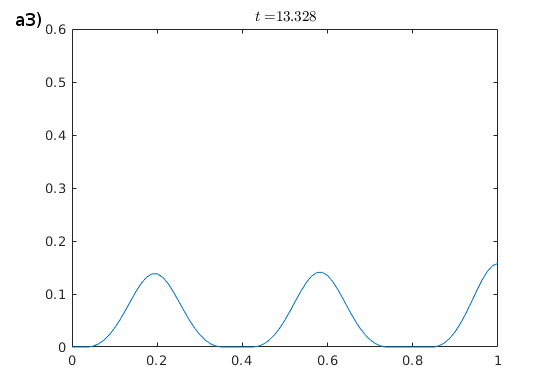}
\\
	\centering
	\includegraphics[scale = 0.33]{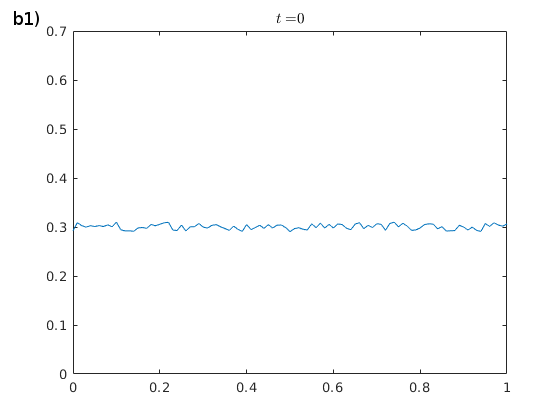}
	\includegraphics[scale = 0.31]{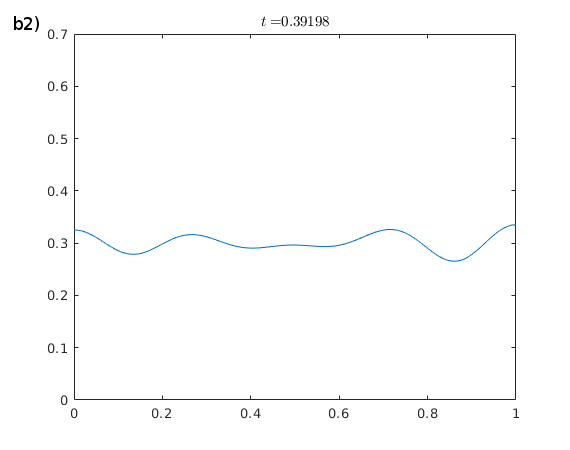}
	\includegraphics[scale = 0.33]{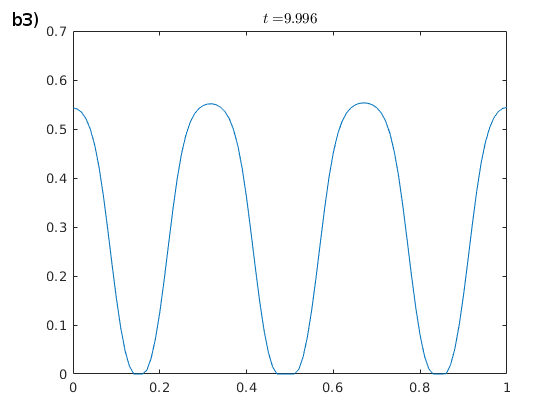}
\\
	\centering
	\includegraphics[scale = 0.33]{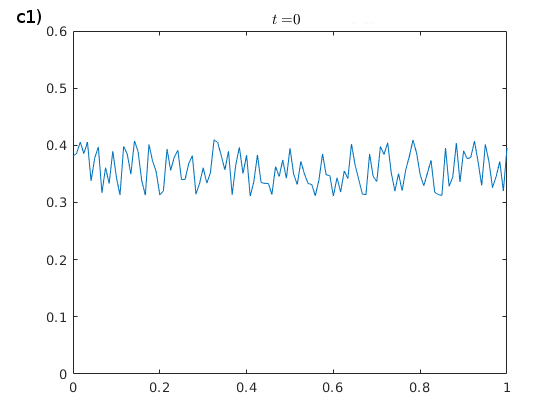}
	\includegraphics[scale = 0.35]{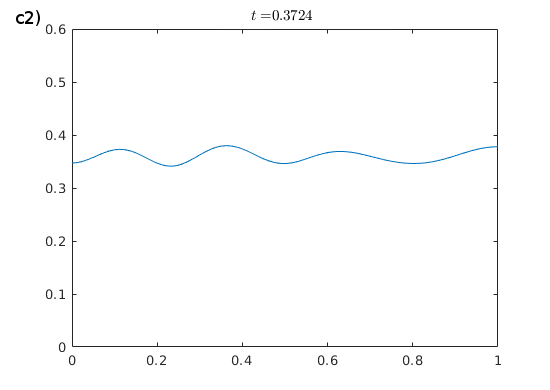}
	\includegraphics[scale = 0.31]{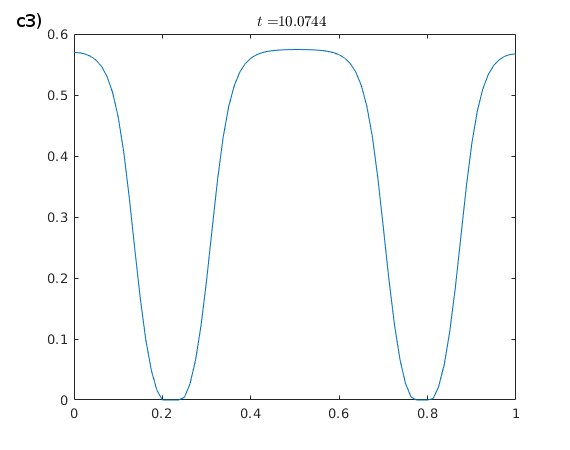}
	\caption{Solution $n_h$ at $3$ different times with $n^0=0.05$(a1,a2,a3), $n^0=0.3$ (b1,b2,b3) and $n^0=0.36$ (c1,c2,c3).}
	\label{fig:sol3steps}
\end{figure}

We can observe that the solution for each of the three test cases remains nonnegative and the mass is conserved throughout the simulations. From Figure~\ref{fig:energy1D}, we observe that the energies decrease monotonically for the three simulations but at different speeds. They all display at the end of the computation a stable (or metastable) state that is a global (or respectively a local) minimizer of the discrete energy.

For the initial condition $n^0=0.3$ (Figures~\ref{fig:sol3steps} \hyperref[fig:sol3steps]{b1)}, \hyperref[fig:sol3steps]{b2)}, \hyperref[fig:sol3steps]{b3)} and Figure~\ref{fig:energy1D} at the middle), the energy decreases rapidly and reaches a plateau showing that the solution evolves rapidly to a steady state. The solution at $t=10$ presents aggregates that are not saturated (i.e. the maximum density is below $n^\star$). The explanation behind this observation is that the initial mass is not sufficiently large for the system to produce any saturated aggregates. However, the clusters appear to be of similar thickness and are relatively symmetrically distributed in the domain. 

For $n^0 = 0.36 > n^\star/2$ (Figures \ref{fig:sol3steps} \hyperref[fig:sol3steps]{c1)}, \hyperref[fig:sol3steps]{c2)}, \hyperref[fig:sol3steps]{c3)}), the aggregates are thicker. The top of the aggregate located at the center of the domain is flatter than for the other simulation. The maximum density is closer to the value $n^\star$ than for the initial condition $n^0=0.3$. Likewise, the symmetry in the domain is respected. Using Figure \ref{fig:energy1D} on the right, we observe that at different times, the energy evolves through several meta-stable equilibria. This reflects the fact that the solution went to different meta-stable states before reaching a stable equilibrium that better minimizes the energy. 

For the initial condition $n^0 = 0.05$ (Figures \ref{fig:sol3steps} \hyperref[fig:sol3steps]{a1)}, \hyperref[fig:sol3steps]{a2)}, \hyperref[fig:sol3steps]{a3)}), the shape of the final solution is different. The aggregates appear to be thinner and far from each other. The symmetry is not retrieved in the domain. Furthermore, from Figure \ref{fig:energy1D} (on the left), we can observe that the evolution of the solution is slow compared to the other initial conditions. The energy seems to be constant in the first moment of the simulation (i.e. after the spinodal decomposition phase). The slow evolution of the solution is explained from the fact that the mobility is degenerate and the amount of mass available in the domain is small. Using Figure \ref{fig:energy1D}, we can also see that the energy continues to decrease even at the end of the simulation. To keep comparable simulation times, we did not reach the complete steady state.


\begin{figure}
	\centering
	\includegraphics[scale=0.34]{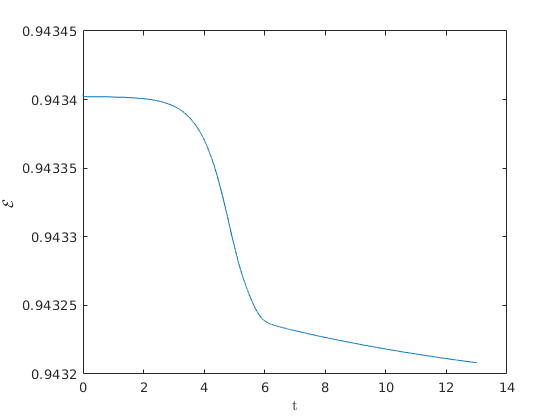}
	\includegraphics[scale=0.32]{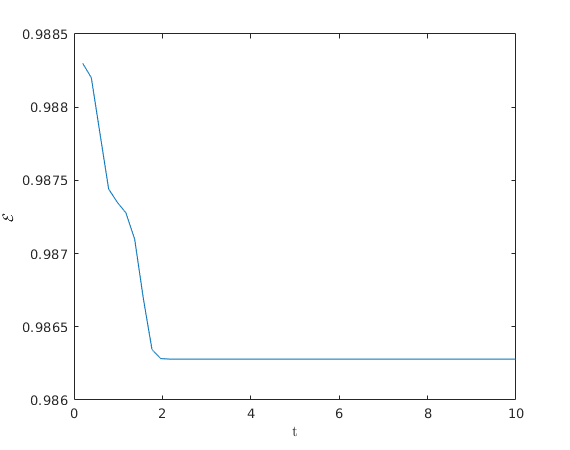}
	\includegraphics[scale=0.32]{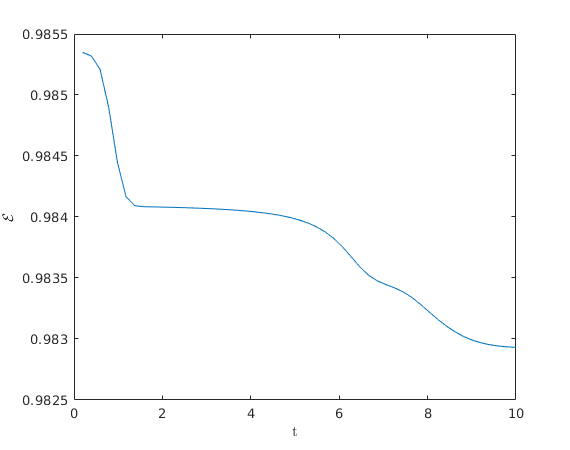}
	\caption{Evolution of discrete energy through time for the 3 initial conditions (from left to right $n^0=\{0.05,0.3,0.36\}$)}
	\label{fig:energy1D}
\end{figure}
Let us compare qualitatively these results with the ones obtained in \cite{agosti_cahn-hilliard-type_2017} for the one dimensional case. For the two test cases $n^0=0.3$ and $n^0=0.36$, there is no differences in the shape the aggregates or the distribution of the mass in the domain.
For $n^0=0.05$, some small discrepancies with the final solutions are observed. In particular, the symmetry of the aggregates in the domain is not respected in our case whereas it is in the reference work. We must stress that doing other simulations, the symmetry was sometimes reached at the time $t\approx 100$ for the initial condition $n^0=0.05$. The reason is that the system will evolve to respect the symmetry but the time at which this stable-steady state is reached depends on the initial distribution of the cell density.

Altogether, the solutions obtained at the end the three simulations are in accordance with the description of the steady-states made in \cite{poulain_relaxation_2019}. The three regions of interest are indeed retrieved at the end of each simulation. 
\subsubsection*{2D test cases}
For the two-dimensional test cases, the domain is a square of length $L=1$. The initial density is computed in the same way as for the one-dimensional test cases i.e. a random uniformly distributed perturbation around $n^0$.
The summary of the values of parameters can be found in Table~\ref{tab:param-2D}.
\begin{table}
	\centering
	\caption{Parameters of the test cases}
	\begin{tabular}{|c|c|}
		\hline
		&Parameters  \\
		\hline
		$\gamma$& $0.014^2$ \\
		$\Delta t$& $2\gamma$\\
		$h$ & $1/64$  \\
		$n^0$ & $[0.05,0.3,0.36]$\\
		$n^\star$& $0.6$ \\
		$\sg$&  $10^{-5}$\\
		\hline
	\end{tabular}
	\label{tab:param-2D}
\end{table}
Figure~\ref{fig:sol3steps-2D} depicts the results of the three test cases with different initial masses. The three simulations respect the nonnegativity of the cell density, the conservation of the initial mass and the monotonic decay of the discrete energy. However, different shapes can be observed for the aggregates. \par
Figures~\ref{fig:sol3steps-2D} \hyperref[fig:sol3steps-2D]{a1)},\hyperref[fig:sol3steps-2D]{a2)},\hyperref[fig:sol3steps-2D]{a3)} show the evolution of the solution through time for the small initial mass ${n^0=0.05}$. Starting from a uniform random distribution of the cell density in the domain, the solution evolves into a more organized configuration. Progressively, a separation of the two phases of the mixture occurs. At the end of the simulation, small clusters are formed. They display a circular shape and are of similar width. The organization of the clusters in the domain tries to maximize the distance between each other. Using the Figure \ref{fig:energy2D} (left), we observe a drop of the energy in the first moments of the simulation denoting a fast reorganization of the randomly distributed initial condition. Then, the solution appears to evolve very slowly, i.e. a meta-stable state was reached. A second drop of the energy follows around $t\approx 15$, the system enters the "coarsening" phase: the small aggregates become more dense and merge with others. At the end, the evolution is very slow. The system continues to rearrange but due to the degeneracy of the mobility and the small amount of initial mass this process is very slow.

\begin{figure}[!ht]
	\centering
	\includegraphics[scale = 0.33]{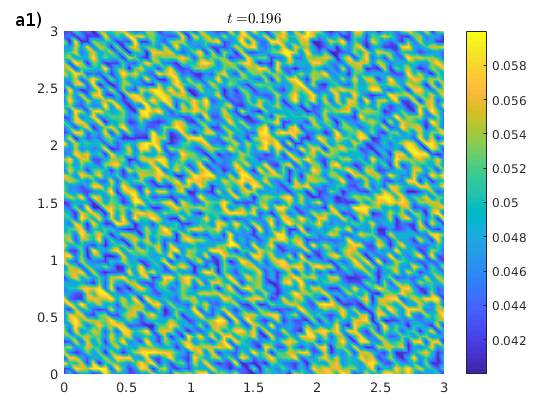}
	\includegraphics[scale = 0.33]{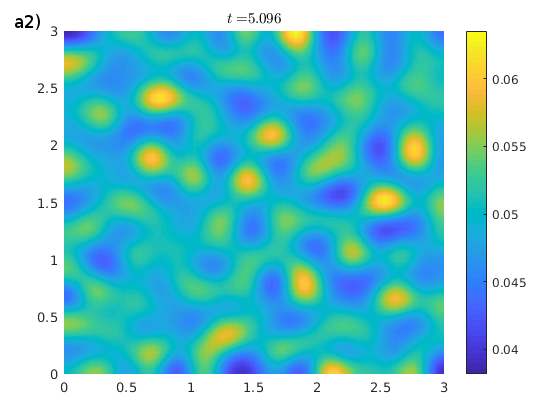}
	\includegraphics[scale = 0.33]{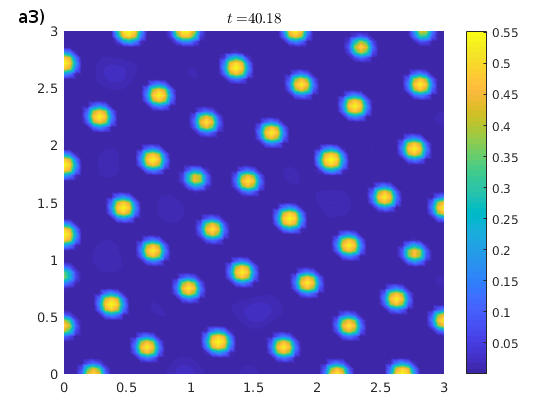}
	\\
	\centering
	\includegraphics[scale = 0.33]{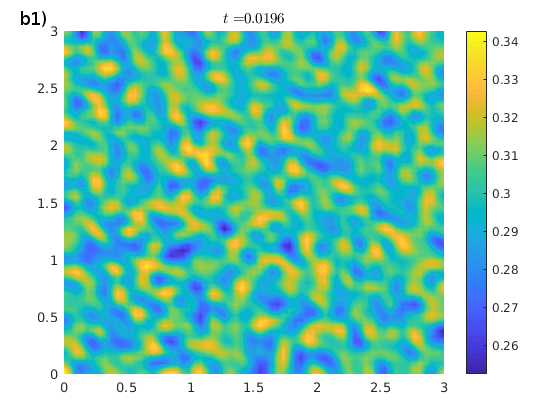}
	\includegraphics[scale = 0.33]{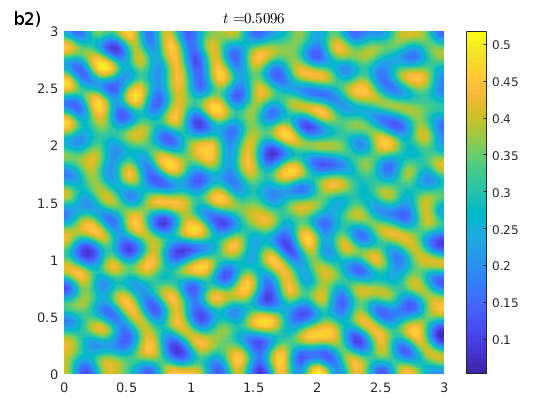}
	\includegraphics[scale = 0.33]{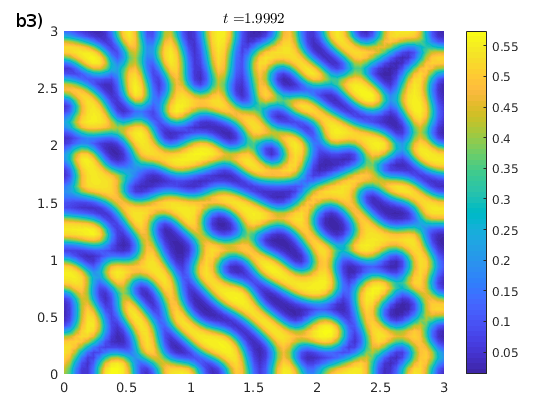}
	\\
	\includegraphics[scale = 0.33]{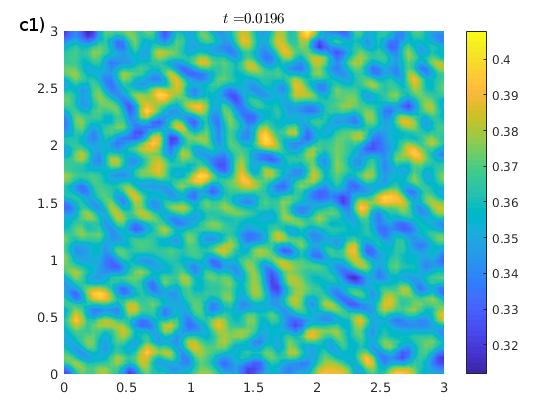}
	\includegraphics[scale = 0.33]{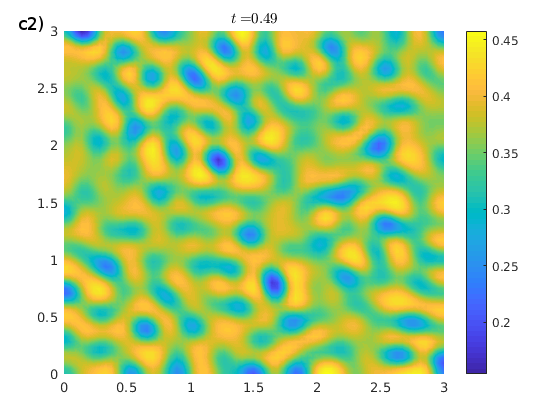}
	\includegraphics[scale = 0.33]{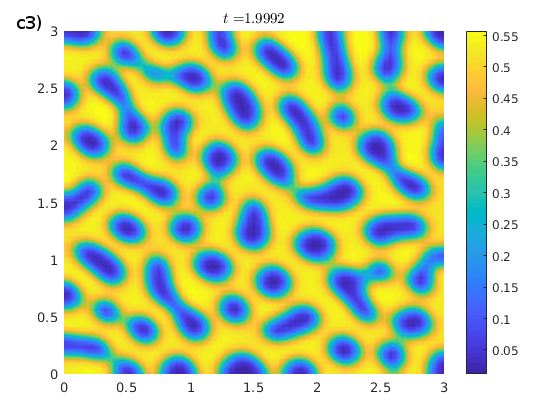}
	\caption{Solution $n_h$ at $3$ different times with $n^0=0.05$ (a1,a2,a3), $n^0=0.3$ (b1,b2,b3) and $n^0=0.36$ (c1,c2,c3).}
	\label{fig:sol3steps-2D}
\end{figure}
Figures \ref{fig:sol3steps-2D} \hyperref[fig:sol3steps-2D]{b1)},\hyperref[fig:sol3steps-2D]{b2)},\hyperref[fig:sol3steps-2D]{b3)} show the evolution of the solution for $n^0=0.3$. 
The two phases that are the spinodal decomposition and the coarsening are retrieved. Between the Figures \ref{fig:sol3steps-2D} \hyperref[fig:sol3steps-2D]{b1)} and \ref{fig:sol3steps-2D} \hyperref[fig:sol3steps-2D]{b2)}, we observe that the solution evolves from a random uniform configuration to an organization in small aggregates that are not saturated. Then (Figure \ref{fig:sol3steps-2D} \hyperref[fig:sol3steps-2D]{b3)}), the cell density is distributed in elongated and saturated aggregates. The separation of the two phases is clear. However, using Figure \ref{fig:energy2D} (middle), we observe that at the end of the simulation the cell density continues to rearrange. Due to the degeneracy of the mobility, this evolution is very slow. 

In Figures \ref{fig:sol3steps-2D} \hyperref[fig:sol3steps-2D]{c1)},\hyperref[fig:sol3steps-2D]{c2)},\hyperref[fig:sol3steps-2D]{c3)}, we can observe the evolution of the solution for $n^0=0.36$. Again, the solution goes through the spinodal decomposition and coarsening phases. 
The only difference that needs to be highlighted for this simulation is the different shape of the aggregates at the end. Indeed, the initial mass being $n^0 = 0.36 > n^\star/2$, the aggregates are wider and more connected to each others. 

Therefore, depending on the initial mass of cells in the domain, the 2D simulations of the model show very different spatial organizations of the cell density.  

\begin{figure}
	\centering
	\includegraphics[scale=0.32]{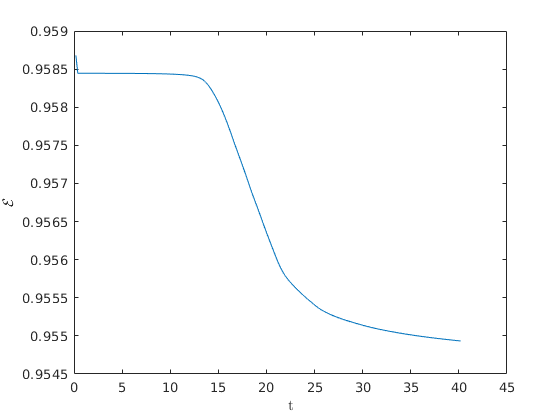}
	\includegraphics[scale=0.32]{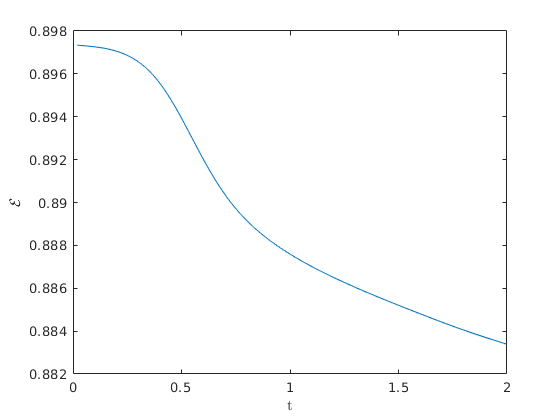}
	\includegraphics[scale=0.32]{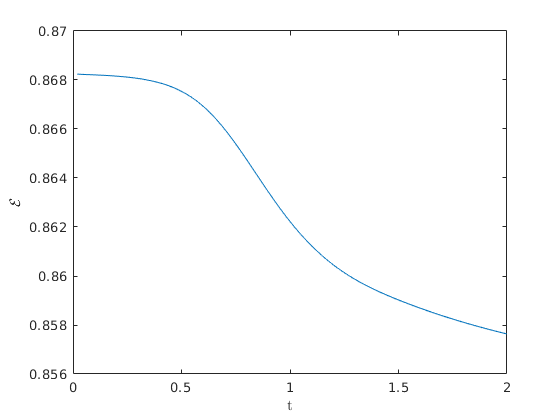}
	\caption{Evolution of discrete energy through time for the 3 initial conditions (from left to right $n^0=\{0.05,0.3,0.36\}$).}
	\label{fig:energy2D}
\end{figure}

Compared to the reference work \cite{agosti_cahn-hilliard-type_2017}, the organizations of the cells for the different initial cell densities are the same. No clear difference can be established regarding the simulation involving the relaxed model and the original one.

The three regions corresponding to a steady-state described in \cite{poulain_relaxation_2019} are retrieved at the end of the simulations for these 2D test cases.

\subsection{Effect of the relaxation parameter \texorpdfstring{$\sg$}{}}
In this section we evaluate the effect of the relaxation parameter $\sg$ for the stability of the scheme, and in particular to satisfy the CFL-like condition \eqref{eq:stab-condition}. This conditions is necessary to preserve the nonnegativity of the solutions of the linear discrete scheme. To evaluate the effect of this parameter on the choice of the time step $\dt$, we compute the amplification matrix $H$ defined by
\[
X^{k+1} = H X^k, \quad \text{with} \quad X^k = \begin{bmatrix} \underline{n}^k \\ \underline{\vp}^k\end{bmatrix}.
\]
Here, $X^k$ is called the state vector.
Using the matrix form of the scheme ~\eqref{eq:numdiscrete-n-expl}--\eqref{eq:numdiscrete-vp-expl} we can decomposed the amplification matrix by $H = H_1^{-1}H_2$ with
\[
H_1 = \begin{bmatrix}  0 & \sigma Q +M_l \\
M_l + \Delta t L^k & \Delta t U^k \end{bmatrix}, \qquad H_2 = \begin{bmatrix} \gamma Q-(1-n^\star)M_l & \frac{\sigma}{\gamma}(1-n^\star)M_l \\
M_l & 0\end{bmatrix}.
\]
We denote by $\lambda_i, i=1,\dots,N_h$, the eigenvalues of the amplification matrix $H$.

To analyze the stability of the numerical scheme due to the relaxation parameter, we compute the spectral radius of the amplification matrix
\[
\rho(H(\dt)) = \max_{i}(|\lambda_i|),
\]
for a smooth initial conditions. 
The scheme is stable when the maximum value of the modulus of the eigenvalues is less or equal to $1$. 
Figure~\ref{fig:stability1D} represents the spectral radius in function of the time step $\dt$ for two values of $\sg$ (the other parameters are the ones taken from the one dimensional test cases with $n^0=0.3$). 
\begin{figure}
     \centering
    \subfloat[$\sg = 10^{-5}$]{\includegraphics[width=0.48\textwidth]{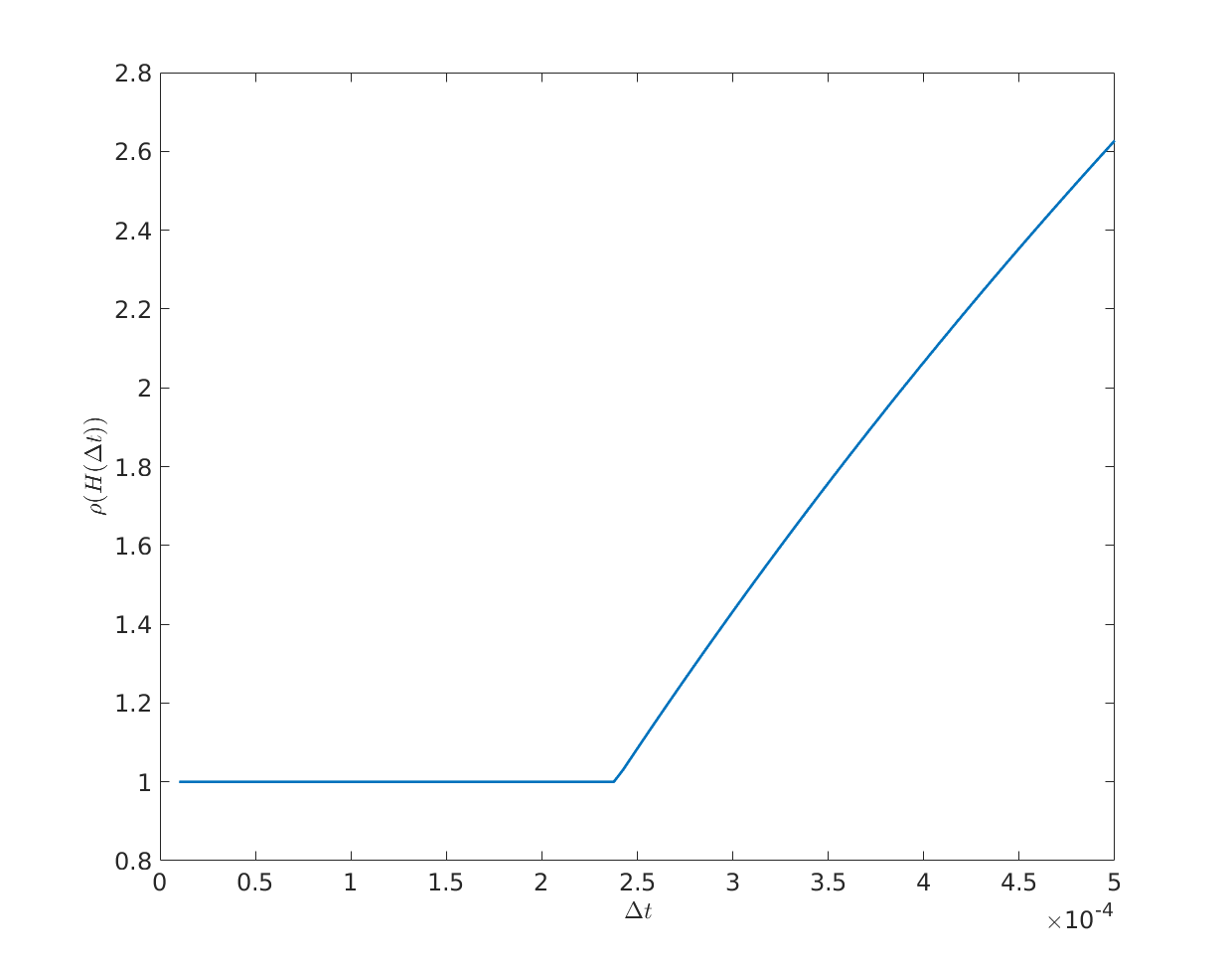}} \hfill
    \subfloat[ $\sg = 10^{-4}$]{\includegraphics[width=0.48\textwidth]{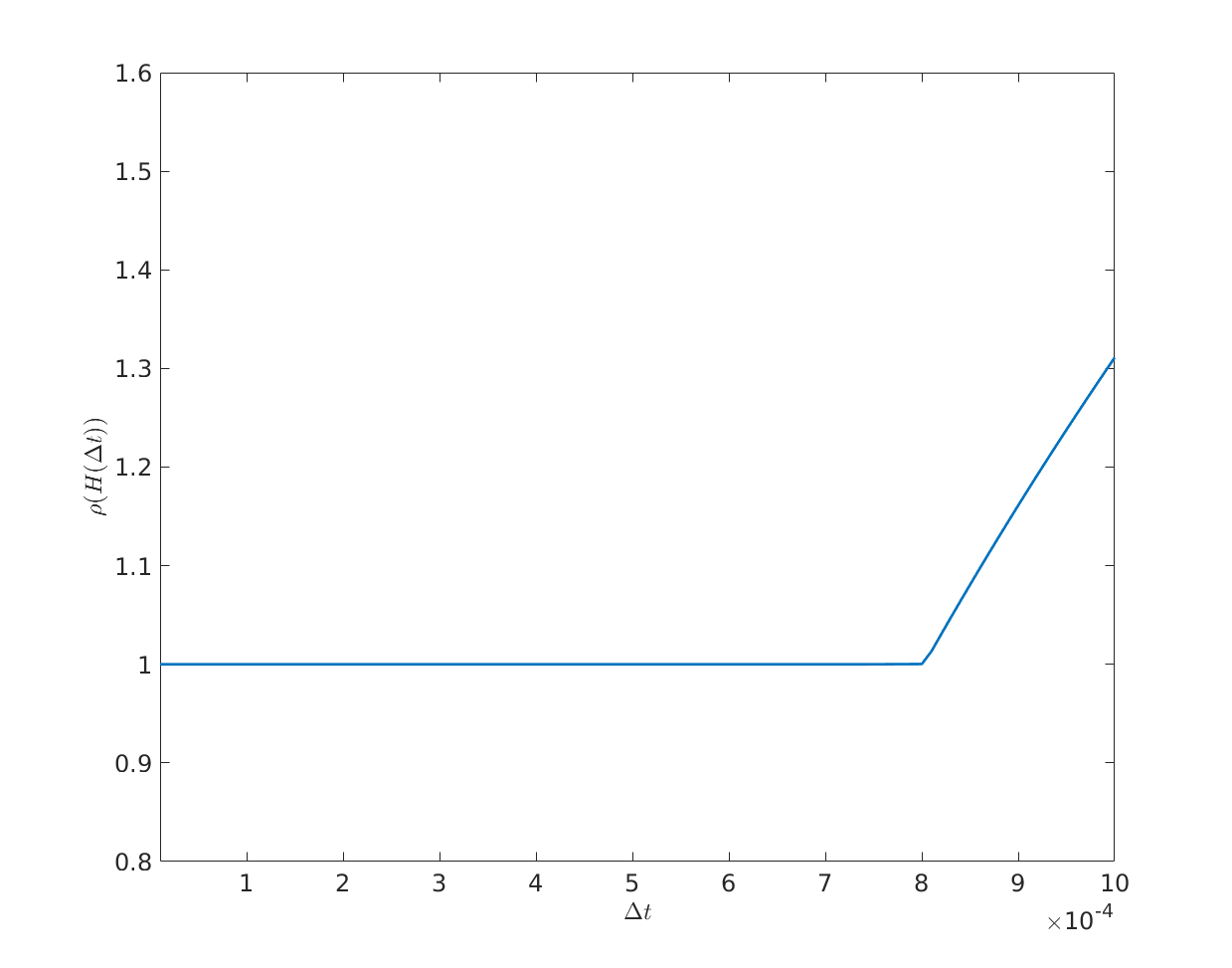}}
	\caption{Spectral radius as a function of $\dt$ for $\sg = 10^{-5}$(left) and $\sg = 10^{-4}$ (right).}
	\label{fig:stability1D}
\end{figure}
We can observe that the scheme remains stable when $\dt$ is small for the two test cases.
However, we see that increasing $\sg$ allows to take larger time steps while remaining stable. This result can be explained due to the fact that increasing $\sg$ diminishes the value 
\[
 \max_{\substack{x_i \in J_h\\ x_j \in \Lambda_i}}\left( \vp_j^\kk-\vp_i^\kk\right),
 \]
  present in the stability condition  \eqref{eq:stab-condition}. Therefore, the regularization induced by the relaxation parameter allows for faster simulations, but it has an effect on the accuracy of the solution compared to the solution given by the non-relaxed model. However,  at the moment it remains unclear how to compare the solution given by a simulation of the relaxed model and a solution of the original degenerate model (without relaxation). This will be the subject of a further work.

\section{Conclusion}
We described and studied a finite element method to solve the relaxed degenerate Cahn-Hilliard equation with single-well logarithmic potential. 
We considered two time discretization schemes: nonlinear semi-implicit and linear semi-implicit.\newline
We showed that the nonlinear scheme preserves the physical properties of the solutions of the continuous model. We proved that this scheme is well-posed and convergent in dimension $d=1,2,3$. The nonnegativity of the solutions is retrieved thanks to the use of an upwind method adapted within the finite element framework. 
The linear semi-implicit scheme allows faster simulations and we proved that it is well-posed and preserves the nonnegativity of the solutions as well. We presented some numerical simulations using this linear scheme in one and two dimensions. The numerical simulations validated the nonnegativity-preserving and energy decaying properties of the scheme. The numerical solutions of the finite element approximation of the RDCH model are in good agreement with previous works dealing with the non-relaxed model. 
We showed that the relaxation parameter $\sigma$ allows us to take a larger time step in the scheme (as long as the condition for the nonnegativity is preserved and $\sigma < \gamma$). 
We point out that thanks to the spatial relaxation, our numerical scheme can be easily implemented and simulations of the relaxed degenerate Cahn-Hilliard model can be computed efficiently using standard softwares. \newline

\section*{Acknowledgements}
A.P. would like to thank Professor Katharina Schratz for useful discussions and valuable comments.

\appendix

\section{Proof of Theorem~\ref{th:conv}}
\label{app:proof}
First, we choose a test function $\chi \in C^\infty_c(\Omega_T,\R^+)$ with $\chi(T,\cdot)=0$. Then, we multiply Equation~\eqref{eq:equation-FV-form} by $\dt_m \chi(x_i,t^k) = \dt_m \chi^k_i$, sum over the $x_i \in J_h$ and over the $k=0,\dots,N_T-1$ to obtain 
\[
    B_m + C_m + D_m + E_m + F_m + G_m = 0,  
\]
with
\begin{align}
    B_m &=  \sum_{k=0}^{N_T-1}\lscal{n_h^\kk - n^k_h}{\chi^k_h},\\
    C_m &= \f\gamma\sg \sum_{k=0}^{N_T-1}  \dt_m \sum_{i=0}^{N_h}\sum_{x_j\in \Lambda_i} \abs{Q_{ij}} \left(\hat B^\kk_{ij}(n^\kk_i-n^\kk_j) - \sqrt{\hat B^\kk_{ij}} (\eta^\kk_i-\eta^\kk_j) \right) (\chi^k_i-\chi_j^k),\\
    D_m &= \f\gamma\sg \sum_{k=0}^{N_T-1}\dt_m \sum_{i=0}^{N_h}\sum_{x_j\in \Lambda_i} \sqrt{\hat B^\kk_{ij}} (\eta^\kk_i-\eta^\kk_j)  (\chi^k_i-\chi_j^k) ,\\
    E_m &= \sum_{k=0}^{N_T-1} \dt_m \sum_{i=0}^{N_h}\sum_{x_j\in \Lambda_i} \abs{Q_{ij}} \left(\widehat{\left(B \psi^{''}_+ \right)}^\kk_{ij}(n^\kk_i-n^\kk_j) - \sqrt{\widehat{\left(B \psi^{''}_+ \right)}^\kk_{ij}} (\zeta^\kk_i-\zeta^\kk_j) \right) (\chi^k_i-\chi_j^k),\\
    F_m &= \sum_{k=0}^{N_T-1} \dt_m  \sum_{i=0}^{N_h}\sum_{x_j\in \Lambda_i}\sqrt{\widehat{\left(B \psi^{''}_+ \right)}^\kk_{ij}} (\zeta^\kk_i-\zeta^\kk_j)  (\chi^k_i-\chi_j^k),\\
    G_m &= -\f{\gamma}{\sigma} \sum_{k=0}^{N_T-1} \dt_m\sum_{i=0}^{N_h}\sum_{x_j \in \Lambda_i}  \tilde B_{ij}^\kk \abs{Q_{ij}} (w_i^\kk-w_j^\kk)(\chi^k_i-\chi_j^k).
\end{align}
Next, we show the limit of each quantity when $m\to\infty$. 

For $B_m$, since $\chi^{N_T}_h = 0$, we have (see~\cite{Eymard2002ConvergenceOA})
\[
    \begin{aligned}
    B_m &= \sum_{k=0}^{N_T-1}\lscal{n_h^\kk }{\chi^k_h}-\sum_{k=1}^{N_T}\lscal{n_h^k }{\chi^k_h} - \lscal{n_h^0 }{\chi^0_h} \\
        &= - \sum_{k=0}^{N_T-1} \dt_m \lscal{n_h^\kk }{\f{\chi^\kk_h-\chi_h^k}{\dt_m}}- \lscal{n_h^0 }{\chi^0_h}.
    \end{aligned}
\]
From the strong convergence~\eqref{eq:strong-conv} and the regularity of $\chi$, we obtain 
\[
    B_m \to   -\int_{\Omega_T} n  \f{\p \chi}{\p t} \,\dd x\,\dd t - \scal{n^0 }{\chi^0},\quad \text{as} \quad m\to\infty.
\]

For $C_m$, we aim to show that $C_m \to 0$ as $m\to \infty$. To do so, we rewrite 
\[
    C_m =  \f\gamma\sg \sum_{k=0}^{N_T-1}  \dt_m \sum_{i=0}^{N_h}\sum_{x_j\in \Lambda_i} \abs{Q_{ij}} \sqrt{\hat B^\kk_{ij}} \left(\sqrt{\hat B^\kk_{ij}} - \sqrt{\ov B^\kk_{ij}} \right) (n^\kk_i-n^\kk_j) (\chi^k_i-\chi_j^k),
\]
where 
\[
    \ov  B^\kk_{ij} = \begin{cases}
        \left(\f{\eta^\kk_i-\eta^\kk_j}{n^\kk_i-n^\kk_j}\right)^2,\quad&\text{if}\quad n^\kk_i\neq n^\kk_j,\\
        \hat B^\kk_i,\quad&\text{if}\quad n^\kk_i= n^\kk_j.
    \end{cases}
\]
From the Cauchy-Schwarz Inequality, we have 
\[
    \begin{aligned}
    \abs{C_m} \le \f\gamma\sg \Big(\sum_{k=0}^{N_T-1}  \dt_m \sum_{i=0}^{N_h}\sum_{x_j\in \Lambda_i} &\abs{Q_{ij}} \hat B^\kk_{ij}(n^\kk_i-n^\kk_j)^2 \Big)^{1/2} \\
    &\times \left(\sum_{k=0}^{N_T-1}  \dt_m \sum_{i=0}^{N_h}\sum_{x_j\in \Lambda_i} \abs{Q_{ij}} \left(\sqrt{\hat B^\kk_{ij}} - \sqrt{\ov B^\kk_{ij}} \right)^2 (\chi^\kk_i-\chi^\kk_j)^2 \right)^{1/2}. 
    \end{aligned}
\] 
The first term on the right-hand side is bounded from \eqref{eq:energy}. The second term, that we denote by $R_m$, is handled using Lemma A.1 in~\cite{cances}. Indeed, we denote for each $K \in \mathcal{T}^{h_m}$, 
\[
    \ov \eta^\kk_K = \max_{x_i\in K}(\eta^\kk_h),\quad \undel \eta^\kk_K = \min_{x_i\in K}(\eta^\kk_h), 
\]
and we define the uniformly continuous function $\sqrt{b \circ \eta }$, defined on the interval $[0,\eta(1)]$, such that for each $K \in \mathcal{T}^{h_m}$,
\[
    \abs{\sqrt{\hat B^\kk_{ij}}- \sqrt{\ov B^\kk_{ij}}}  \le \omega (\ov \eta^\kk_K - \undel \eta^\kk_K),\quad \forall x_j\in K\cap \Lambda_i.
\]   
Therefore, we have 
\[
    \begin{aligned}
    R_m = \sum_{k=0}^{N_T-1}  &\dt_m \sum_{i=0}^{N_h}\sum_{x_j\in \Lambda_i} \abs{Q_{ij}} \left(\sqrt{\hat B^\kk_{ij}} - \sqrt{\ov B^\kk_{ij}} \right)^2 (\chi^\kk_i-\chi^\kk_j)^2  \\
    &\le  \sum_{k=0}^{N_T-1}  \dt_m    \sum_{K\in\mathcal{T}^{h_m}}\left( \omega (\ov \eta^\kk_K - \undel \eta^\kk_K)\right)^2  \sum_{i=0}^{N_h}       \sum_{x_j\in \Lambda_i} \abs{Q_{ij}} (\chi^\kk_i-\chi^\kk_j)^2 .
    \end{aligned}
\]
From the regularity of $\chi$, we know that $ \sum_{i=0}^{N_h}       \sum_{x_j\in \Lambda_i} \abs{Q_{ij}} (\chi^\kk_i-\chi^\kk_j)^2 $ is bounded and we obtain 
\[
    R_m \le C \int_{\Omega_T} \omega  (\ov \eta^\kk_K - \undel \eta^\kk_K) \,\dd x\,\dd t.
\]
From Lemma A.1 in~\cite{cances}, we obtain $R_m \to 0, \quad \text{as}\quad m\to \infty,$ hence, $C_m \to 0$. Using similar arguments, we have $E_m \to 0$, as $m\to\infty$.

We now turn to the convergence of $D_m$. To this end, we define two piecewise constant quantities on each element $K\in \mathcal{T}^{h_m}$
\[
    \Theta_{h_m,\dt_m} = \sqrt{b \circ \eta^{-1}}(\Gamma_{h_m,\dt_m}), \quad \text{and} \quad \Gamma_{h_m,\dt_m}(t,x) = \eta_{h_m,\dt_m}(t,x_T),\quad \forall x\in K,\quad t\in (t^k,t^\kk].    
\]
Then, we define
\[
    \begin{aligned}
    D'_m &= \f\gamma\sigma\int_{\Omega_T} \Theta_{h_m,\dt_m}\nabla \eta_{h_m,\dt_m} \cdot\nabla \chi_{h_m,\dt_m}\,\dd x\,\dd t,\\
    &=  \f\gamma\sigma\sum_{k=0}^{N_T-1}  \dt_m \sum_{K\in \mathcal{T}^{h_m}} \Theta_{h_m,\dt_m} \sum_{x_i\in K} \sum_{x_j\in K \cap \Lambda_i} \abs{Q_{ij}} (\eta^\kk_i-\eta^\kk_j)  (\chi^k_i-\chi_j^k).
    \end{aligned}
\]

From a generalization of Lemma A.1 in~\cite{cances}, and the boundedness of the two previous quantities, we have, as $m\to\infty$, 
\[
    \begin{aligned}
        \Theta_{h_m,\dt_m} \to \sqrt{b(n)},\quad\text{in}\quad L^2(\Omega_T),\\
        \Gamma_{h_m,\dt_m} \to \eta(n),\quad\text{in}\quad L^2(\Omega_T),
    \end{aligned}
\]
from which we conclude 
\[
    D'_m \to \f\gamma\sigma\int_{\Omega_T} \sqrt{b(n)}\nabla \eta(n) \cdot \nabla \chi\,\dd x\,\dd t,\quad \text{as}\quad m\to \infty.  
\]
However, we need to show that $\abs{D_m-D'_m}\to 0$ as $m\to \infty$. 
We use similar arguments as to show that $C_m \to 0$. Indeed, we rewrite
\[
    D_m-D'_m = \f\gamma\sigma\sum_{k=0}^{N_T-1}  \dt_m \sum_{i =0}^{N_h} \sum_{K\in \mathcal{T}^{h_m}: x_i\in K} \sum_{x_j\in K \cap \Lambda_i} \abs{Q_{ij}} \left(\sqrt{\hat B^\kk_{ij}} - \sqrt{ a^\kk_K} \right) (\eta^\kk_i-\eta^\kk_j)  (\chi^k_i-\chi_j^k)
\]
with $\sqrt{ a^\kk_K} = \Theta^\kk_K$, we are in position to repeat the arguments presented for the convergence of $C_m \to 0$ and we do not repeat them here. Therefore, we obtain, as $m\to\infty$,
\[
    D_m \to \f\gamma\sigma\int_{\Omega_T} \sqrt{b(n)}\nabla \eta(n) \cdot\nabla \chi\,\dd x\,\dd t.
\]

The same arguments are also applied to the convergence of $F_m$ to obtain, as $m\to\infty$, 
\[
    F_m \to \int_{\Omega_T} \sqrt{b(n)\psi^{''}_{+}(n)}\nabla \zeta(n) \cdot\nabla \chi\,\dd x\,\dd t. 
\]

The last term of the first equation to analyze is $G_m$. We repeat again the same arguments. We use the previously defined quantities $\Theta_{h_m,\dt_m}$ and $\Gamma_{h_m,\dt_m}$. Then, we define 
\[
    \begin{aligned}
    G'_m &= \f\gamma\sigma\int_{\Omega_T} (\Theta_{h_m,\dt_m})^2\nabla w_{h_m,\dt_m} \cdot\nabla \chi_{h_m,\dt_m}\,\dd x\,\dd t,\\
    &=  \f\gamma\sigma\sum_{k=0}^{N_T-1}  \dt_m \sum_{K\in \mathcal{T}^{h_m}} (\Theta_{h_m,\dt_m})^2 \sum_{x_i\in K} \sum_{x_j\in K \cap \Lambda_i} \abs{Q_{ij}} (w^\kk_i-w^\kk_j)  (\chi^k_i-\chi_j^k),
    \end{aligned}
\]
and we have, as $m\to \infty$, 
\[
    G'_m \to\f\gamma\sigma \int_{\Omega_T} b(n) \nabla w \cdot\nabla \chi\,\dd x\,\dd t.
\]

We still have to prove that $D_m-D'_m\to 0$ as $m\to \infty$.
Since we have 
\[
    \abs{\tilde B^\kk_{ij} - (\Theta_{h_m,\dt_m})^2} \le C\omega^2\abs{\ov \eta^\kk_K - \undel \eta^\kk_K}^2,  
\]
we obtain the result using similar arguments as for the convergence of $R_m$.
In the limit $m\to \infty$, we obtained the first equation of the weak system defined in~\eqref{def:solution-RDCH}. 

For the second equation, we obtain the limit equation using the weak convergence of $w_{h_m,\dt_m}$ given by~\eqref{eq:weak-conv}, the strong convergence of $w_{h_m,\dt_m}$ given by~\eqref{eq:strong-conv}, the continuity of $\psi_-'(\cdot)$, and the strong convergence of $n_{h_m,\dt_m}$ given by~\eqref{eq:convergence-n}.

This concludes the proof of Theorem~\ref{th:conv}.
\bibliography{biblio}
\bibliographystyle{siam}
\end{document}